\documentclass[a4paper,11pt]{amsart}
\usepackage{amsmath}
\usepackage{amsthm,amsfonts,amssymb,graphics}
\usepackage{pgfplots}
\usepackage[foot]{amsaddr}

\usepackage{hyperref}
\usepackage[normalem]{ulem}
\usepackage{enumerate}
\usepackage{geometry}
\usepackage{dsfont}
\usepackage{centernot}
\usepackage{bbm}

\newtheorem{theorem}{Theorem}[section]
\newtheorem{lemma}[theorem]{Lemma}
\newtheorem*{algorithm*}{Template algorithm $\mathcal{T}$ for an event $A$}
\newtheorem{proposition}[theorem]{Proposition}
\newtheorem{corollary}[theorem]{Corollary}

\theoremstyle{remark}
\newtheorem{remark}[theorem]{Remark}

\numberwithin{equation}{section}

\newcommand\C{\mathcal{C}}
\newcommand \id{\mathds 1}
\newcommand {\R} {\mathbb{R}}
\newcommand {\E} {\mathbb{E}}
\newcommand {\N} {\mathbb{N}}
\newcommand {\Z} {\mathbb{Z}}
\renewcommand {\P} {\mathbb{P}}
\newcommand {\Vol} {\textrm{Vol}}
\newcommand {\PVol} {\widetilde{\textrm{Vol}}}
\newcommand {\dist} {\textrm{dist}}
\newcommand {\eps} {\varepsilon}
\newcommand {\cone} {\textrm{Cone}}

\begin{document}
\title[Mean-field bounds for Poisson-Boolean percolation]{Mean-field bounds for Poisson-Boolean percolation}
\author{Vivek Dewan$^1$}
\email{vivek.dewan@univ-grenoble-alpes.fr}
\address{$^1$Institut Fourier, Universit\'{e} Grenoble Alpes}
\author{Stephen Muirhead$^2$}
\email{smui@unimelb.edu.au}
\address{$^2$School of Mathematics and Statistics, University of Melbourne}
\begin{abstract}
We establish the mean-field bounds $\gamma \ge 1$, $\delta \ge 2$ and $\triangle \ge 2$ on the critical exponents of the Poisson-Boolean continuum percolation model under a moment condition on the radii; these were previously known only in the special case of fixed radii (in the case of $\gamma$), or not at all (in the case of $\delta$ and $\triangle$). We deduce these as consequences of the mean-field bound $\beta \le 1$, recently established under the same moment condition [8], using a relative entropy method introduced by the authors in previous work [7].
\end{abstract}
\date{\today}
\thanks{}

\maketitle

\section{Introduction}
The behaviour of critical statistical physics models is believed to be described by a set of \textit{critical exponents} which govern the scaling of macroscopic observables at, or near, criticality (see, e.g., \cite[Chapter 9]{gri99} for an introduction to critical exponents for Bernoulli percolation). In general these exponents depend on the dimension of the ambient space, but they are expected to assume a \textit{mean-field value} if the dimension exceeds the \textit{upper-critical dimension}, conjectured to be $d_c = 6$ for percolation. For many of the critical exponents, but not all, the low-dimension values are bounded by their mean-field value; these are known as \textit{mean-field bounds}.

\smallskip 
In this paper we consider mean-field bounds for the \textit{Poisson-Boolean model}, a continuum percolation model introduced in \cite{gil61, hal85} that is defined as the union of Euclidean balls centred at the points of a Poisson point process on $\R^d$ with \textit{intensity} $\lambda > 0$, whose radii are independently drawn from a \textit{radius distribution} $\mu$ supported on $\R_+$. Equivalently, we can define Poisson-Boolean percolation as
\begin{equation}
\label{e:pbp}
\mathcal{O}  = \bigcup_{(x, r) \in \eta} \{ x + B_r \} , 
\end{equation}
where $\eta$ is a Poisson point process on $\R^d \times \R_+$ with intensity $\lambda dx \otimes d\mu$, $dx$ is the Lebesgue measure on $\R^d$, and $B_r \subset \R^d$ denotes the Euclidean ball of radius $r$ centred at $0$. We think of the radius distribution $\mu$ as being fixed and the intensity $\lambda$ as variable, and write $\P_\lambda$ for the law of $\mathcal{O}$ with intensity $\lambda$, with corresponding expectation $\E_\lambda$. We refer to \cite{pen03,mr08} for background on this model, and see \cite{att18,drt20} for a selection of recent results.

\smallskip
We say two sets $X, Y \subset \R^d$ are \textit{connected} if there exists a path in $\mathcal{O}$ between $X$ and $Y$, and we denote this event by $\{X \longleftrightarrow Y\}$ (with the standard abuse of notation if $X$ or $Y$ are singletons). The \textit{infinite cluster density} of the model is defined as
\[ \theta(\lambda) = \P_\lambda[ 0 \longleftrightarrow \infty] = \lim_{R \to \infty}  \P_\lambda[ 0 \longleftrightarrow \partial B_R]  ,\]
and the \textit{critical parameter} of the model is 
\[ \lambda_c = \lambda_c(\mu) = \inf\{ \lambda  \ge 0 : \theta(\lambda) > 0 \}  \in [0, \infty] . \]
Although a priori one could have $\lambda_c \in \{0, \infty\}$, it is known that the integrability condition
\begin{equation}
\label{a:wk}
 \int r^d d\mu(r)<\infty
 \end{equation}
is both necessary and sufficient for $\lambda_c \in (0, \infty)$ \cite{hal85, gou08}. We assume \eqref{a:wk} for the remainder of the paper; this is without loss of generality, since if \eqref{a:wk} is violated then $\mathcal{O} = \R^d$ almost surely so the model is trivial.

 \subsection{Critical exponents}
As mentioned above we are interested in the \textit{critical exponents} of the model which we now define. We denote by $\C$ the connected component (or `cluster') of $\mathcal{O}$ that contains the origin (setting $\C = \emptyset$ if $0 \notin \mathcal{O}$), with $\Vol(\C)$ its volume. It is believed that both the density $\theta(\lambda)$, $\lambda > \lambda_c$, and the \textit{susceptibility} $\chi(\lambda)   =  \E_\lambda[  \Vol( \C ) ]$, $\lambda < \lambda_c$, have power law behaviour as, respectively, $\lambda \downarrow \lambda_c$ and $\lambda \uparrow \lambda_c$, and also that the cluster volume $\Vol(\C)$ has a power law tail at $\lambda = \lambda_c$. Hence it is natural to define critical exponents $\beta$, $\gamma$ and $\delta$ via
\[  \theta(\lambda) = (\lambda-\lambda_c)^{\beta + o(1) }  ,  \quad  \lambda \downarrow \lambda_c  ,   \]
\[ \chi(\lambda) = (\lambda_c - \lambda)^{-\gamma + o(1) } ,  \quad  \lambda \uparrow \lambda_c  , \] 
and
\[ \P_{\lambda_c} [ \Vol( \C ) > y ]  = y^{-1/\delta + o(1)} , \quad  y \to \infty ,   \] 
whenever these exponents exist. It is further expected that, for each $k \in \mathbb{N}$,
\[  \E_\lambda[  \Vol( \C )^k ] = (\lambda_c - \lambda)^{-\gamma - \triangle (k-1)+ o(1) }   ,  \quad  \lambda \uparrow \lambda_c ,   \]
and we define the `gap exponent' $\triangle$ whenever it exists.

\smallskip
By analogy with Bernoulli percolation, it is natural to expect that these critical exponents exist and satisfy the mean-field bounds
\begin{equation}
\label{e:mfb}
 \beta \le 1 \ , \quad \gamma \ge 1 \ , \quad \delta \ge 2 \quad \text{and} \quad \triangle \ge 2  , 
 \end{equation}
and that these bounds are saturated above the upper-critical dimension $d \ge d_c = 6$. In the case of Bernoulli percolation the bounds \eqref{e:mfb} are classical \cite{cc87,an84,ab87,dn85}, and the fact that they are saturated in sufficiently high dimension has also been established \cite{an84, hs90, fh17}.

\smallskip
For Poisson-Boolean percolation much less is known about the critical exponents. It was recently shown \cite{drt20} that under the stronger moment condition
\begin{equation}
\label{a:str}
 \int r^{5d-3} d\mu(r)<\infty ,
 \end{equation}
 the mean-field density lower bound 
 \begin{equation}
 \label{e:mfbbeta}
 \theta(\lambda) \ge c(\lambda - \lambda_c ) 
 \end{equation}
 holds for some $c = c(\mu) > 0$ and all $\lambda > \lambda_c$ sufficiently close to $\lambda_c$. This implies $\beta \le 1$ if this exponent exists, and extends previous results that established \eqref{e:mfbbeta} in the case of bounded radii \cite{lpz17,zie18}. The mean-field bounds on $\gamma$, $\delta$ and $\triangle$ have not yet been established at this level of generality. Indeed to our knowledge only the inequality $\gamma \ge 1$ is known, and only for fixed radius \cite{hhlm19} (note however that \cite{hhlm19} used a slightly different definition of $\gamma$; see the discussion in Section \ref{s:vol} below). In the fixed radius case it is further known that $\gamma = 1$ in sufficiently high dimension~\cite{hhlm19}.
 
 \subsection{Main results}
Our main result establishes the mean-field bounds \eqref{e:mfb} under the assumption~\eqref{a:str}:

\begin{theorem}[Mean-field bounds]
\label{t:mfb}
Assume \eqref{a:str} and suppose that the exponents $\gamma$, $\delta$ and $\triangle$ exist. Then
\[ \gamma \ge 1 \ , \quad \delta \ge 2 \quad \text{and} \quad \triangle \ge 2 . \]
\end{theorem}

The mean-field bounds stated above are conditional in the sense that they assume the existence of the critical exponents $\gamma, \delta$ and $\triangle$. We next present unconditional versions of these bounds, which also demonstrate that Theorem \ref{t:mfb} is a consequence of the mean-field bound \eqref{e:mfbbeta} in full generality. In the following results we do not assume \eqref{a:str} (this condition is relevant to us only as a sufficient condition to ensure \eqref{e:mfbbeta}).

\begin{theorem}[Bounds on the susceptibility]
\label{t:s}
For every $\lambda_0 \in (0, \lambda_c)$ there exists $c = c(\mu, \lambda_0)$ such that, for all $\lambda \in (\lambda_0, \lambda_c)$,
\begin{equation}
\label{e:s1}
 \chi(\lambda) \ge c (\lambda_c - \lambda)^{-2} \theta(2 \lambda_c - \lambda)  
 \end{equation}
and
\begin{equation}
\label{e:s2}
 \chi(\lambda) \ge c (\lambda_c - \lambda)^{-2} \P_{\lambda_c}[ \Vol(\C) \ge (\lambda_c-\lambda)^{-2} ]  .
 \end{equation}
\end{theorem}

\begin{theorem}[Bounds on the critical cluster volume]
\label{t:t}
Suppose there exist $c_0, \beta_0 > 0$ and $\lambda_1 > \lambda_c$ such that
\[  \theta(\lambda) \ge c_0(\lambda - \lambda_c )^{\beta_0}  , \quad \text{for all } \lambda \in( \lambda_c , \lambda_1) .\]
 Then there exists a $c = c(\mu,c_0,\beta_0,\lambda_1)$ such that, for all $y \ge 1$,
\begin{equation}
\label{e:t1} 
\P_{\lambda_c}[ \Vol(\C) \ge y  ]^{2/\beta_0 - 1} \int_0^y \P_{\lambda_c}[ \Vol(\C) \ge u  ] du \ge c .  
\end{equation}
Moreover for every $\lambda_0 \in (0, \lambda_c)$ there exists a $c = c(\mu,\lambda_0)$ such that, for all $\lambda \in (\lambda_0, \lambda_c)$ and $y \ge 1$,
\begin{equation}
\label{e:t2}
 \log \frac{ \P_{\lambda}[ \Vol(\C) \ge y  ] }{ \P_{\lambda_c}[ \Vol(\C) \ge y  ] }  \ge - 1 - c  (\lambda_c - \lambda)^2 ( \P_{\lambda_c}[ \Vol(\C) \ge y  ] )^{-1}  \int_0^y \P_{\lambda_c}[ \Vol(\C) \ge u  ] du       .
 \end{equation}
\end{theorem}

To state the final set of unconditional bounds we introduce the \textit{(critical) magnetisation}
\[ M(\rho) =\E_{\lambda_c}\big[1 - e^{-\rho \Vol(\mathcal{C})} \big]   , \quad \rho > 0,\]
following the pioneering approach to Bernoulli percolation in \cite{ab87}. One can define an associated critical exponent $\delta_M$ via
\[ M(\rho) = \rho^{1/\delta_M + o(1)} , \quad  \rho \downarrow 0  ,\]
if such an exponent exists, and by general properties of the Laplace transform one can verify that $\delta = \delta_M$ if both exponents exist.

\begin{theorem}[Bounds on the magnetisation]
\label{t:m}
For every $\lambda_0 < \lambda_c$ there exists a $c = c(\mu, \lambda_0)$ such that, for all $\lambda \in (\lambda_0, \lambda_c)$,
\begin{equation}
\label{e:m1}
 \chi(\lambda)  \ge  c  (\lambda_c - \lambda)^{-2}  M( ( \lambda_c - \lambda)^2) . 
 \end{equation}
Suppose in addition there exist $c_0, \beta_0 > 0$ and $\lambda_1 > \lambda_c $ such that
\[ \theta(\lambda) \ge c_0(\lambda - \lambda_c )^{\beta_0}  , \quad \text{for all } \lambda \in( \lambda_c , \lambda_1) . \]
 Then there exists a $c = c(\mu, c_0, \beta_0, \lambda_1)$ such that, for all $\rho \in (0,1]$,
  \begin{equation}
\label{e:m2} 
M(\rho) \ge c \rho^{\beta_0/2} . 
\end{equation}
\end{theorem}

\begin{remark}
If we assume the existence of the exponents $\beta, \gamma$ and $\delta$, the results in Theorems~\ref{t:s}--\ref{t:m} imply the inequalities
\begin{equation}
\label{e:ineq}
 \gamma \ge 2 - \beta \ , \quad \gamma \ge 2 - 2/\delta \quad \text{and} \quad \delta \ge 2/\beta ,
 \end{equation}
which can be deduced, respectively, from \eqref{e:s1}, \eqref{e:s2} (or alternatively \eqref{e:t2} or \eqref{e:m1}) and \eqref{e:t1} (or alternatively \eqref{e:m2}); these inequalities are saturated if the exponents take their mean-field values $\beta = \gamma = 1$ and $\delta = 2$. While the inequalities \eqref{e:ineq} are classical in the context of Bernoulli percolation \cite{new86,new87,ab87}, to our knowledge they were not yet known for any dependent percolation model in general dimension $d \ge 2$. Even for Bernoulli percolation, our method to obtain these inequalities is novel and quite different from classical methods, since it does not rely on the analysis of differential inequalities.

\smallskip
Interestingly, to our knowledge the bound \eqref{e:t2} is new even in the context of Bernoulli percolation. Under the assumption that the exponent $\delta$ exists, this bound takes the form (see \eqref{e:mfb2} for a precise statement)
\[   \P_{\lambda}[ \Vol(\C) \ge y  ]  \gtrsim y^{-1/\delta} \exp(  - c (\lambda_c - \lambda)^2 y   ) ,      \]
and under the standard percolation near-critical scaling hypothesis (see \cite[Chapter 9]{gri99}) it is natural to expect that this bound is sharp, up to constants in the exponent, for $d > d_c = 6$ (this has recently been proven in sufficiently high dimension \cite{hut20}, and see also \cite{chs21,hms21} for related `near-critical scaling' results for Bernoulli percolation in sufficiently high dimension).
\end{remark}

To finish the section we verify that Theorem \ref{t:mfb} is a consequence of the above bounds:

\begin{proof}[Proof of Theorem \ref{t:mfb} assuming Theorems \ref{t:s}--\ref{t:m}]
Since we assume \eqref{a:str}, the bound on the density \eqref{e:mfbbeta} is available. Hence \eqref{e:s1} implies that $\chi(\lambda) \ge c (\lambda_c-\lambda)^{-1}$ for some $c > 0$ and $\lambda$ sufficiently close to $\lambda_c$, from which $\gamma \ge 1$ follows. 

To prove $\delta \ge 2$ we observe that if $\delta$ exists then, for every $\varepsilon > 0 $ and $y \ge y_0 = y_0(\varepsilon) \ge 1$,
\begin{equation}
\label{e:mfb0}
  y^{-1/\delta  - \varepsilon} \le  \P_{\lambda_c}[ \Vol(\C) \ge y  ]   \le y^{-1/\delta  + \varepsilon} .  
  \end{equation}
Hence \eqref{e:t1} implies that (taking $\beta_0 = 1$ since \eqref{e:mfbbeta} is available), for $y \ge y_0$,
\begin{equation}
\label{e:mfb1}
c \le \P_{\lambda_c}[ \Vol(\C) \ge y  ]  \int_0^y \P_{\lambda_c}[ \Vol(\C) \ge u  ]  \le y^{-1/\delta + \varepsilon} \Big( \int_{y_0}^y u^{-1/\delta  + \varepsilon}  du + y_0 \Big)    . 
\end{equation}
If $\delta < 1$ then taking $\varepsilon < 1/\delta - 1$ the right-hand side of \eqref{e:mfb1} tends to zero as $y \to \infty$, which is a contradiction. On the other hand if $\delta \ge 1$ then the right-hand side of \eqref{e:mfb1} is at most, for sufficiently large $y$,
\[     y^{-1/\delta + \varepsilon} \Big( (1-1/\delta+\varepsilon)^{-1} y^{1-1/\delta + \varepsilon} + y_0  \Big) \le  (1-1/\delta + \varepsilon/2)^{-1}  y^{1-2/\delta + 2\varepsilon} .\]
Sending $y \to \infty$ we deduce that $1-2/\delta + 2\varepsilon \ge 0$, and sending $\varepsilon \to 0$ yields that $\delta \ge 2$.

To establish $\triangle \ge 2$ we assume that $\delta \ge 2$ exists and use the bounds \eqref{e:mfb0} to deduce from \eqref{e:t2} that, for every $\varepsilon > 0 $,  $\lambda_0 < \lambda_c$, $\lambda \in (\lambda_0, \lambda_c)$, and $y \ge y_0$ for some $y_0 = y_0(\varepsilon) \ge 1$,
\begin{align}
\nonumber \P_{\lambda}[ \Vol(\C) \ge y  ]  &  \ge  c_1  y^{-1/\delta - \varepsilon } \exp \Big(  - c_2 (\lambda_c - \lambda)^2 y^{1/\delta + \varepsilon}   y^{1- 1/\delta + \varepsilon }    \Big) \\
\label{e:mfb2} & =   c_1  y^{-1/\delta -\varepsilon} \exp \Big(  - c_2 (\lambda_c - \lambda)^2 y^{1+2\varepsilon}    \Big) ,
\end{align}
for constants $c_1, c_2  > 0$ depending only on $\mu$ and $\lambda_0$. Then for every $k \in \N$, $\lambda \in (\lambda_0, \lambda_c)$, and $\varepsilon > 0$,
\[ \E_\lambda[  \Vol( \C )^k   ] = k  \int_{0}^\infty y^{k-1}  \P_{\lambda}[ \Vol(\C) \ge y  ] dy \ge k c_1 \int_{y_0}^\infty y^{(k-1)-1/\delta -\varepsilon } e^{-c_2(\lambda_c-\lambda)^2 y^{1+2\varepsilon } } dy . \]
Changing variables to $y' = (\lambda_c-\lambda)^2 y^{1+2\varepsilon}$, we deduce that \begin{align}
\nonumber \E_\lambda[  \Vol( \C )^k   ] & \ge   k  c_1 (\lambda_c-\lambda)^{-2/(1+2\varepsilon) \times (1-1/\delta + (k-1) - \varepsilon )} \int_{\lambda_c^2 y_0^{1+2\varepsilon}}^\infty (y')^{1/(1+2\varepsilon)\times(k-1/\delta -\varepsilon)-1 } e^{-c_2 y' } dy' \\
\label{e:mfb3} &  = c_3  (\lambda_c-\lambda)^{-2/(1+2\varepsilon) \times (1-1/\delta + (k-1)  - \varepsilon)} 
\end{align}
for all $\varepsilon > 0$, some $c_3 = c_3(\mu, \lambda_0, \delta, \varepsilon, k)$, and all $\lambda \in (\lambda_0, \lambda_c)$.
On the other hand, since we assume that the exponents $\gamma$ and $\triangle$ also exist, for each $k \in \N$,
\begin{equation}
\label{e:mfb4}
\E_\lambda[  \Vol( \C )^k   ]  =  (\lambda_c-\lambda)^{-\gamma - \triangle (k-1) + o(1) }  , \quad \lambda \uparrow \lambda_c . 
\end{equation}
Sending $\lambda \uparrow \lambda_c$ in \eqref{e:mfb3}, comparing with \eqref{e:mfb4}, and then sending $\varepsilon \to 0$, shows that
\[    \gamma + \triangle(k-1) \ge 2 - 2/\delta + 2(k-1)  , \]
which gives a contradiction for sufficiently large $k$ unless $\triangle \ge 2$.
\end{proof}

\subsection{Other possible definitions of the exponents}
\label{s:vol}
In certain applications it may be more natural to consider other definitions of the critical exponents $\gamma$, $\delta$ and $\triangle$. Our proof can be adapted in some cases but not all.

\subsubsection{Interpreting `volume' as the number of balls that comprise the cluster}
If the Poisson-Boolean model is considered as a random graph on the projection of $\eta$ onto $\R^d$ with connections induced by intersecting balls, it could be more natural to define the exponents relative to the \textit{number of balls that comprise the cluster of the origin}, $\# \C  =  |  (x, r) \in \eta : x \in  \C |$, rather than the Euclidean volume $\Vol(\C)$; this is how the exponent $\gamma$ is defined in the work \cite{hhlm19} mentioned above. In the case that the radius distribution $\mu$ is bounded it is straightforward to adapt our proof to show that the mean-field bounds hold for the exponents $\gamma$, $\delta$ and $\triangle$ defined in this way -- essentially this is due to the universal volume comparison $\# \C \ge c \Vol(\C)$ for some $c = c(\mu) > 0$. However this comparison is no longer valid if $\mu$ is unbounded, and we do not know whether the mean-field bounds hold in full generality in that case.

\subsubsection{Percolation of the vacant set} One could also consider the dual percolation problem associated with the \textit{vacancy set} $\mathcal{V}$ (the closure of the complement of $\mathcal{O}$) and define the associated critical exponents $\gamma$, $\delta$ and $\triangle$ relative to the volume of the component of $\mathcal{V}$ that contains the origin. Once again, in the case that $\mu$ is bounded we believe that our proof can be adapted to establish the mean-field bounds on $\gamma$, $\delta$ and $\triangle$, but in general there are additional complications (again the issue lies in establishing the relevant `volume comparison', analogous to Section \ref{s:vc1}). We leave the general case as an open problem.
 
\subsection{About the proof}
The proof of Theorems \ref{t:s}--\ref{t:m} makes use of a `relative entropy method' that can be summarised as follows. Combining exploration arguments with properties of the relative entropy, one can efficiently bound from above the relative entropy, for different $\lambda$, between the laws of certain well-chosen observables of $\mathcal{O}$ (such as $\Vol(\C)$). By applying this to $\lambda$ ranging between slightly subcritical to slightly subcritical, and combining with general entropic bounds (e.g.\ Pinsker's inequality), one can deduce information on critical exponents. 

\smallskip
This strategy was introduced by the authors in a previous work \cite{dm21}, where it was used to give bounds on the one-arm exponent for Bernoulli percolation and other Gaussian dependent percolation models. It was subsequently exploited in \cite{hut21} where it used to show, in the context of Bernoulli percolation, that mean-field behaviour of the susceptibility implies mean-field behaviour of the infinite cluster density and critical cluster volume.

\smallskip
In this paper we show how to adapt this method to the Poisson-Boolean model, where certain extra technical difficulties arise. As well as the applications in the current paper, we expect that similar applications to those developed in \cite{dm21, hut21} could also be implemented for the Poisson-Boolean model using this approach. We expect this method can also be adapted to other models in the Bernoulli universality class, such as Poisson-Voronoi percolation.

\subsection{Acknowledgements}
The second author was supported by the Australian Research Council (ARC) Discovery Early Career Researcher Award DE200101467. The authors thank Tom Hutchcroft for helpful comments, and also an anonymous referee for valuable and detailed feedback on an earlier version of this paper.

\medskip

\section{The entropic bounds}
\label{s:entropic}

In this section we establish the main entropic bounds (see Proposition \ref{p:entropic}) that will underpin the proof of Theorems \ref{t:s}--\ref{t:m}. These are similar to bounds proven in \cite{dm21} in the setting of Bernoulli and Gaussian percolation.

\smallskip
Recall that the Poisson-Boolean model $\mathcal{O}$ can be defined via \eqref{e:pbp} as a function of the Poisson point process $\eta$ on $\R^d \times \R_+$ with intensity $\lambda dx \otimes d\mu$. We now introduce the concept of a \textit{randomised algorithm} for the point process $\eta$. Let $C = (C_i)_{i \in \Z^d \times \Z_+}$ denote the set of hypercubes $C_i = i +  [0, 1]^d \times [0,1]$ which partition $\R^d \times \R_+$ up to boundaries. A \textit{(randomised) algorithm} $\mathcal{A}$ is an adapted procedure which sequentially reveals $\eta|_{C_{i_j}}$ for a random sequence of hypercubes $(C_{i_j})_{j \in \N}$, terminates at a (possibly infinite) stopping time $\tau \in \mathbb{N} \cup \{\infty\}$, and upon termination returns a value in $\{0, 1\}$. We denote by $\mathcal{W}_\mathcal{A}= \cup_{1 \le j \le \tau} C_{i_j}$ the union of the hypercubes revealed by $\mathcal{A}$. 

\smallskip
For an event $A$ measurable with respect to $\eta$, an algorithm is said to \textit{locally determine $A$} if, for every $\lambda > 0$, the following hold almost surely:
\begin{itemize}
\item For every $n \in \N$ there exists a (deterministic) finite subset $D_n \subset \Z^d \times \Z_+$, independent of $\lambda > 0$, such that $\cup_{1 \le j \le n} C_{i_j} \subset D_n$;
\item If the algorithm terminates in finite time, it returns the value $\id_A$;
\item On the event $A$, the algorithm terminates in finite time.
\end{itemize} 
In words, this means that the algorithm has only a bounded number of choices at each step, and if the event $A$ occurs the algorithm must verify this in finite time. Note that this allows the algorithm to never terminate if $A$ does not occur, which is convenient in our setting since to determine $\mathcal{O}$ on any set requires determining $\eta$ on an infinite number of hypercubes.

\smallskip
For a Borel subset $S \subset \R^d \times \R_+$, we write $\PVol(S)$ to denote its `Poisson-Boolean volume'
\[ \PVol(S) =  \int_S   dx \otimes d\mu  . \] 
Note that $\PVol (\mathcal{W}_\mathcal{A}) $ may be finite even if $\mathcal{A}$ does not terminate.

\begin{proposition}[Entropic bounds]
\label{p:entropic}
Let $A$ be an event and let $\mathcal{A}$ be an algorithm that locally determines $A$. Then for all $\lambda_1, \lambda_2 > 0$,
\[ |\P_{\lambda_1}[A]-\P_{\lambda_2}[A]|\leq  \frac{|\lambda_2-\lambda_1|}{\sqrt{\lambda_2}}\sqrt{2\max(\P_{\lambda_1}[A],\P_{\lambda_2}[A]) \E_{\lambda_1} [ \PVol (\mathcal{W}_\mathcal{A}) ] }  ,\]
and if $\P_{\lambda_1}[A]$ and $\P_{\lambda_2}[A]$ are both non-zero then
\[ \log \P_{\lambda_1}[A] -   \log \P_{\lambda_2}[A]  \le    \frac{ \lambda_2^{-1} (\lambda_2-\lambda_1)^2 \E_{\lambda_1} [ \PVol (\mathcal{W}_\mathcal{A}) ]  }{ \P_{\lambda_1}[A] }  + 1  .\]
\end{proposition}

\begin{remark}
Our definition of an algorithm that `locally determines' an event is analogous to the notion of \textit{Borel computation} introduced in \cite[Section 3]{hut21} in the context of Bernoulli percolation.
\end{remark}

\begin{remark}
The bounds in Proposition \ref{p:entropic} are in terms of the Poisson-Boolean volume $\PVol(\mathcal{W}_\mathcal{A})$ of the revealed set, rather than volume in the ambient space $\R^d$. In Section \ref{s:vc} we show that, for the events we consider, we can convert between these two volumes up to multiplicative constants outside an event of negligible probability.

Related to this, instead of exploring $\eta$ restricted to unit hypercubes, one might try to work with $\eps$-scale hypercubes and take the limit $\eps \to 0$ so as to explore $\eta$ `in the continuum'. While Proposition \ref{p:entropic} itself does not depend on $\eps$, a difficulty arises when translating the bound in terms of volume in the ambient space, where the conversion degenerates with $\eps$.
\end{remark}

To analyse the magnetisation $M(\rho)$ we need to extend these bounds slightly. For $\rho > 0$, let $\mathcal{G}$ denote a Poisson point process on $\R^d$ with intensity $\rho$, independent of $\mathcal{O}$, with $\P_{\lambda,\rho}$ and $\E_{\lambda,\rho}$ denoting the joint law of $(\mathcal{O}, \mathcal{G})$ and corresponding expectation; we refer to $\mathcal{G}$ as the \textit{ghost field}. The relevance of $\mathcal{G}$ to the magnetisation is that, by conditioning on $\mathcal{O}$ and using the independence of $\mathcal{G}$ and $\mathcal{O}$,
\begin{equation}
    \label{e:mag}
 \P_{\lambda_c, \rho}[0 \longleftrightarrow \mathcal{G} ] = \E_{\lambda_c}\big[     \P_{\lambda_c, \rho}[ | \C \cap \mathcal{G}| \ge 1    \, | \,  \mathcal{O} ]  \big] =  \E_{\lambda_c}\big[1 - e^{-\rho \Vol(\mathcal{C})} \big]  =: M(\rho) .
 \end{equation} 
In the presence of $\mathcal{G}$, a (randomised) algorithm $\mathcal{A}$ is an adapted procedure which first reveals $\mathcal{G}$ and then sequentially reveals $\eta$ for a random sequence of hypercubes $(C_{i_j})_{1 \le j \le \tau}$, with $\mathcal{W}_\mathcal{A}$ defined as before. We extend the definition of the algorithm `locally determining' an event in the natural way.

\begin{proposition}[Entropic bounds with the ghost field]
\label{p:entropic2}
Let $A$ be an event that depends on $(\mathcal{G}, \eta)$ and let $\mathcal{A}$ be an algorithm that locally determines $A$. Then for all $\lambda_1, \lambda_2 > 0$ and $\rho>0$, the conclusion of Proposition \ref{p:entropic} holds with $\P_{\lambda_i}$ and $\E_{\lambda_i}$ replaced by $\P_{\lambda_i, \rho}$ and $\E_{\lambda_i, \rho}$ respectively.
\end{proposition}

The remainder of the section is devoted to proving Propositions \ref{p:entropic} and \ref{p:entropic2}. For this we first recall some basic properties of the relative entropy, and then present a `stopping time lemma' that is the main ingredient in the proof. 
 
\subsection{Basic properties of the relative entropy}
\label{s:basic}
Let $P$ and $Q$ be probability measures defined on a common measurable space. The \textit{relative entropy (or Kullback-Leibler divergence) from $P$ to $Q$} is defined as
\[   D_{KL}(P || Q) := \int \log \Big( \frac{dP}{dQ} \Big) \, dP  \]
if $P$ is absolutely continuous with respect to~$Q$, and $D_{KL}(P || Q) := \infty$ otherwise; $D_{KL}(P || Q)$ is non-negative by Jensen's inequality. If $X$ and $Y$ are random variables taking values in a common measurable space, with respective laws $P$ and $Q$, we also write $D_{KL}(X || Y)$ for $D_{KL}(P||Q)$. We recall two basic properties of the relative entropy (see e.g \cite[Theorem 2.2 and Corollary 3.2]{kul78} and \cite[Theorem D.13]{dem10}):
\begin{enumerate}
\item(Chain rule) Let $X = (X_1, X_2)$ and $Y = (Y_1, Y_2)$ be random variables on a common product of Borel spaces. Then 
\begin{equation}
\label{e:chain}
D_{KL}(X || Y) = D_{KL}(X_1 || Y_1) +  \mathbb{E}_{x \sim X_1} \big[  D_{KL}( (X_2 | X_1 = x)  || (Y_2 | Y_1 = x ) ) \big] . 
\end{equation}
If particular if $X_2$ and $Y_2$ are identically distributed and independent of $X_1$ and $Y_1$ respectively, then
\begin{equation}
\label{e:chain2}
D_{KL}(X || Y) = D_{KL}(X_1 || Y_1)  . 
\end{equation}
\item(Contraction) Let $X$ and $Y$ be random variables taking values in a common measurable space and let $F$ be a measurable map from that space. Then 
\begin{equation}
\label{e:contract}
 D_{KL}( F(X) || F(Y) )  \le  D_{KL}( X || Y )   .
 \end{equation}
\end{enumerate}

We shall also make use of the following general bounds involving the relative entropy:

\begin{lemma}
\label{l:bounds}
Let $P$ and $Q$ be probability measures on a common measurable space and let $A$ be an event. Then
\[ |P(A) - Q(A)| \le  \sqrt{ 2 \max\{ P(A), Q(A) \} D_{KL}(P \| Q)} , \]
and if $P[A]$ and $Q[A]$ are both non-zero then
\[\log P[A]  -  \log Q[A] \le  \frac{ D_{KL}( P || Q) }{P[A] }  + 1    .\]
\end{lemma}
\begin{proof}
The first statement is a variant of Pinsker's inequality, and is proven in \cite[Lemma 2.12]{dm21}. For the second statement, by Jensen's inequality
\begin{align*}
-D_{KL}(P \| Q)  = \int  \log \frac{dQ}{dP} dP   &= \int_A \log \frac{dQ}{dP} dP  + \int_{A^c} \log \frac{dQ}{dP} dP   \\
&\le P[A] \Big( \log \frac{Q[A]}{P[A]}  + \frac{P[A^c]}{P[A]}   \log \frac{Q[A^c]}{P[A^c]}  \Big) \\
& \le  P[A] \Big( \log \frac{Q[A]}{P[A]}  + 1\Big) ,
\end{align*}
where the last inequality is since $Q[A^c] \le 1$ and $\sup_{x \in [0, 1]} \frac{1-x}{x} \log \frac{1}{1-x} = 1$, which is easily verified. Rearranging gives the result.
\end{proof}

Finally we need a bound for the relative entropy between Poisson point processes with proportionate intensities:
\begin{lemma}
\label{l:poisson}
Let $(D, \mathcal{A}, \mu)$ be a finite Borel space, and let $X$ and $Y$ be Poisson point processes on $D$ with respective intensities $\lambda_X d\mu$ and $\lambda_Y d\mu$ with $\lambda_X, \lambda_Y > 0$. Then 
\[ D_{KL}(X || Y)\leq \mu(D)  \frac{(\lambda_Y-\lambda_X)^2}{\lambda_Y}. \]
\end{lemma}
\begin{proof}
Let $U$ and $V$ be Poisson random variables with respective parameter $\lambda_X \mu(D)$ and $\lambda_Y \mu(D)$, and let $S = (S_i)_{i \ge 1}$ be an i.i.d.\ sequence of random variables in $D$ with distribution proportional to $\mu$. Then $X \stackrel{d}{=} \sum_{i \le U} \delta_{S_i}$ and $Y \stackrel{d}{=} \sum_{i \le V} \delta_{S_i}$, where $\delta$ is a unit Dirac mass, so that $X$ (resp. $Y$) is a measurable function of $(U,S)$ (resp.\ $(V,S)$). Further, $U$ (resp.\ $V$) and $S$ are both defined on Borel spaces. Hence by the contraction property \eqref{e:contract} and the chain rule \eqref{e:chain2},
\[ D_{KL}(X || Y)  \le D_{KL}\big( (U, S)  ||  (V, S) \big)  =    D_{KL}\big( \textrm{Pois}(\lambda_X \mu(D) ) ||  \textrm{Pois}(\lambda_Y \mu(D) ) \big) . \]
A simple computation shows that 
\[  D_{KL}\big( \textrm{Pois}(\lambda_1)  ||  \textrm{Pois}(\lambda_2) \big) =    \lambda_2-\lambda_1 +\lambda_1 \log ( \lambda_1 / \lambda_2)  \]
for any $\lambda_1, \lambda_2 > 0$. Hence using the bound $\log x \le x-1$, valid for all $x > 0$, we have 
\begin{equation*}
 D_{KL}(X || Y)  \leq \mu(D) \Big( \lambda_Y-\lambda_X +\lambda_X \Big(\frac{\lambda_X-\lambda_Y}{\lambda_Y} \Big) \Big)  = \mu(D)  \frac{(\lambda_Y-\lambda_X)^2}{\lambda_Y}   . \qedhere
\end{equation*}
\end{proof}

\subsection{The stopping time lemma}
The proof of Proposition \ref{p:entropic} relies on an exact formula for the relative entropy between stopped sequences of independent random variables; this is a generalisation of \cite[Lemma 2.11]{dm21} which considered the i.i.d.\ case.

\smallskip
Fix $n \ge 1$ and a collection $(E_i)_{1 \le i \le n}$ of Borel spaces, and consider the product space $E = E_1 \times \cdots \times E_n$. A \textit{decision tree} $\sigma$ with \textit{stopping time} $\tau$ is a measurable mapping $x \in E \mapsto (\sigma_i(x))_{1 \le i \le n}, \tau) \in \{1, \ldots , n\}^n \times \{1, \ldots , n\}$ such that (i) $\sigma_k(x)$ and $\{\tau(x) \ge k\}$ are determined by $(x_{\sigma_i(x)})_{i \le k-1}$, and (ii) $\sigma_i(x)\neq \sigma_j(x)$ for all $j<i$. We define the corresponding \textit{stopped sequence} $x^\tau$ as $x^\tau_i = x_{\sigma_i(x)}$ for $i \le \tau$, and $x^\tau_i = \dagger$ for $i > \tau$, where $\dagger$ is an arbitrary symbol.

\begin{lemma}
\label{l:stl}
Let $X= (X_i)_{1 \le i \le n}$ and $Y = (Y_i)_{1 \le i \le n}$ be sequences of independent random variables such that $X_i$ and $Y_i$ take values in $E_i$ with respective laws $\nu_{X,i}$ and $\nu_{Y,i}$ which are mutually absolutely continuous. Let $\sigma$ be a decision tree with stopping time $\tau$, and let $X^\tau$ and $Y^\tau$ be the corresponding stopped sequences. Then
\[ D_{KL} \big(   X^\tau \big\|  Y^\tau \big)  =  \E \Big[\sum\limits_{k\leq \tau(X)} d(\sigma_k(X)) \Big] ,\]
where $d(i) = D_{KL}( \nu_{X,i}  \|  \nu_{Y,i}  )$.
\end{lemma}

\begin{proof}
Define $X^{k \wedge \tau} =  (X^\tau_i)_{i \le k}$ and analogously for $Y$. By the chain rule \eqref{e:chain}, for each $k \ge 1$,
\begin{align*}
& D_{KL} \big(   X^{(k+1) \wedge \tau} \big\|   Y^{(k+1) \wedge \tau} \big) \\
&  \quad  =  D_{KL} \big(   X^{k \wedge \tau}  \big\| Y^{k \wedge \tau}  \big)   \! +  \E_{x \sim (X^\tau_i)_{i \le k}} \!  \big[ D_{KL} \! \left( X^\tau_{k+1} \big| (X^\tau_i)_{i \le k} \! = \! x  \big\| Y^\tau_{k+1}  \big| (Y^\tau_i)_{i \le k}  \! = \! x \right)  \! \big] \\
& \quad =  D_{KL} \big(   X^{k \wedge \tau}  \big\| Y^{k \wedge \tau}  \big)   \! \\
&\quad \quad +  \E_{x \sim (X^\tau_i)_{i \le k}} \!  \Big[  \id_{\tau(X) \ge k+1} D_{KL}  \! \left( X^\tau_{k+1} \big| (X^\tau_i)_{i \le k} \! = \! x  \big\| Y^\tau_{k+1}  \big| (Y^\tau_i)_{i \le k}  \! = \! x \right)  \!\Big]  .
\end{align*}
Since $\sigma_{k+1}(x)\neq \sigma_j(x)$ for all $j\leq k$, $X_{\sigma_{k+1}(X)}$ depends on $(X^\tau_i)_{i \le k}$ only through $\sigma_{k+1}(X)$. Hence 
\[  D_{KL}  \! \left( X^\tau_{k+1} \big| (X^\tau_i)_{i \le k} \! = \! x  \big\| Y^\tau_{k+1}  \big| (Y^\tau_i)_{i \le k}  \! = \! x \right) =  d(\sigma_{k+1}(x)) , \]
 and so by induction,
\begin{equation*}
 D_{KL} \big(   X^\tau \big\|   Y^\tau \big)  =   \sum_{1 \le k \le n}\E\Big[ \id_{\tau(X) \ge k} d(\sigma_k(X)) \Big]   =   \E\Big[\sum\limits_{k\leq \tau(X)} d(\sigma_k(X)) \Big].    \qedhere
\end{equation*}
\end{proof}

\subsection{Proof of Propositions \ref{p:entropic} and \ref{p:entropic2}}
First we prove Proposition \ref{p:entropic}. It is convenient to work with a truncation of the algorithm $\mathcal{A}$ and the event $A$. Precisely, for $n \in \N$, we define $\mathcal{A}_n$ to be the algorithm that explores the same hypercubes as $\mathcal{A}$ except terminates at $\tau_n = \min\{\tau, n\}$ and returns value $0$ if $\tau > n$, and define also the event $A_n = A \cap \{\tau \le n \}$. Since $\mathcal{A}$ locally determines $A$, the truncated algorithm and event satisfies the following properties for every $\lambda > 0$:
\begin{itemize}
\item For every $n \in \N$, $\mathcal{A}_n$ explores a subset of a deterministic finite set $D_n \subset \Z^d \times \Z_+$ of hypercubes and terminates in bounded time $\tau_n \le n$ with value $\id_{A_n}$;
\item $\P_{\lambda}[A_n] \to \P_\lambda[A]$ as $n \to \infty$; and 
\item For every $n \in \N$, $\E_{\lambda} [ \PVol (\mathcal{W}_{\mathcal{A}_n}) ] \le \E_{\lambda} [ \PVol (\mathcal{W}_{\mathcal{A}}) ]$.
\end{itemize}
Now let $X = (X_i)_{i \in D_n}$ and $Y = (Y)_{i \in D_n}$ be sequences of independent point processes such that $X_i$ and $Y_i$ are distributed as $\eta|_{C_i}$ under $\P_{\lambda_1}$ and $\P_{\lambda_2}$ respectively. Since $X_i \stackrel{d}{=} \sum_{j \le N_i} \delta_{S_j^i}$, where $N_i \stackrel{d}{=} \textrm{Pois}(\eta(C_i))$ and $(S^i_j)_j$ is an i.i.d.\ sequence of random variables on $C_i$ with distribution proportional to $\eta|_{C_i}$, we can view $X_i$ as a random variable taking values in a Borel space. In particular, the results in Section \ref{s:basic} apply to $X$ and $Y$. 

The algorithm $\mathcal{A}_n$ defines a decision tree $\sigma$ and stopping time $\tau$ for the sequences $X$ and $Y$, and since $\mathcal{A}_n$ determines the value of $\id_{A_n}$, by the contraction property \eqref{e:contract} the relative entropy between the law of $\id_{A_n}$ under $\P_{\lambda_1}$ and $\P_{\lambda_2}$ is at most $D_{KL} \big(   X^\tau \big\|   Y^\tau \big)$.
By Lemma \ref{l:poisson} we have, for any $i \in D_n$,
\begin{equation}
\label{e:entropic1}
 d(i) = D_{KL}(X_i || Y_i)  \le \frac{(\lambda_2-\lambda_1)^2}{\lambda_2} \PVol (C_i)  .  
 \end{equation}
 Applying Lemma~\ref{l:stl} and \eqref{e:entropic1} we have
\begin{align}
 \nonumber D_{KL} \big(   X^\tau \big\|   Y^\tau \big)   & = \E_{\lambda_1} \Big[\sum\limits_{k\leq \tau(X)}d( \sigma_k(X)  )\Big]  \le    \E_{\lambda_1}\Bigg[\sum_{k \le \tau(X)} \frac{(\lambda_2-\lambda_1)^2}{\lambda_2} \PVol(C_{\sigma_k(X)}) \Bigg]  \\
 \label{e:entropic2} & = \frac{(\lambda_2-\lambda_1)^2}{\lambda_2} \E_{\lambda_1}\Big[  \PVol \big( \cup_{k \le \tau(X)}   C_{\sigma_k(X)} \Big) \Big] =\frac{(\lambda_2-\lambda_1)^2}{\lambda_2}\E_{\lambda_1} [\PVol (\mathcal{W}_{\mathcal{A}_n} ) ] .
 \end{align}
 Combining with Lemma \ref{l:bounds} yields the bounds 
\[ |\P_{\lambda_1}[A_n]-\P_{\lambda_2}[A_n]|\leq  \frac{|\lambda_2-\lambda_1|}{\sqrt{\lambda_2}}\sqrt{2\max(\P_{\lambda_1}[A_n],\P_{\lambda_2}[A_n]) \E_{\lambda_1} [ \PVol (\mathcal{W}_{\mathcal{A}_n}) ] }  \]
and
\[ \log \P_{\lambda_1}[A_n] -   \log \P_{\lambda_2}[A_n]  \le    \frac{ \lambda_2^{-1} (\lambda_2-\lambda_1)^2 \E_{\lambda_1} [ \PVol (\mathcal{W}_{\mathcal{A}_n}) ]  }{ \P_{\lambda_1}[A_n] }  + 1  ,\]
and taking $n \to \infty$ establishes Proposition \ref{p:entropic}.

For Proposition \ref{p:entropic2} we redefine $X = (X_i)_{i \in D_n \cup \{0\}}$ and $Y = (Y_i)_{i \in D_n \cup \{0\}}$ where $X_0$ and $Y_0$ are independent point processes distributed as $\mathcal{G}$ under $\P_{\lambda_1,\rho}$ and $\P_{\lambda_2,\rho}$. Since 
\[  d(0) = D_{KL}(X_0 || Y_0) = 0  ,    \]
the same argument that led to \eqref{e:entropic2} gives
\begin{align*}
 D_{KL} \big(   X^\tau \big\|   Y^\tau \big) &    \le    \E_{\lambda_1,\rho}\Bigg[\sum_{k \le \tau(X)} \id_{\sigma_k(X) \neq 0} \frac{(\lambda_2-\lambda_1)^2}{\lambda_2} \PVol(C_{\sigma_k(X)}) \Bigg]  \\
 &  =\frac{(\lambda_2-\lambda_1)^2}{\lambda_2}\E_{\lambda_1,\rho} [\PVol (\mathcal{W}_{\mathcal{A}_n} ) ]  .
\end{align*}  
Combining with Lemma \ref{l:bounds} and taking $n \to \infty$ completes the proof as before

\medskip

\section{Volume comparison}
\label{s:vc}

In this section we establish volume comparison results which allow us to convert Poisson-Boolean volume, as appears in the entropic bounds from the previous section, into volume in the ambient space which is relevant for the main results.

\subsection{From the percolation volume to the number of intersecting cubes}
\label{s:vc1}

We first compare the volume of a subset of the Poisson-Boolean percolation $\mathcal{O}$ to the number of unit cubes that it intersects.

\smallskip
Let $S = (S_i)_{i \in \Z^d}$ denote the set of cubes $S_i = i +  [0, 1]^d $ which partition $\R^d$ up to boundaries (we call these `cubes' to distinguish them from the  `hypercubes' $C_i \subset \R^d \times \R_+$ that were introduced in Section \ref{s:entropic}). For a subset $D \subset \R^d$, possibly random, denote by $\mathcal{S}_D$ the collection of cubes that intersect $D$. 
 
 We introduce the collection $\mathfrak{O} = \{  \bigcup_{(x, r) \in \eta'} \{ x + B_r \} : \eta'\subset \eta \}$ of all possible subsets of balls which comprise the Poisson-Boolean percolation $\mathcal{O}$. In particular $\mathcal{C} \in \mathfrak{O}$.
 
\begin{proposition}
\label{p:vc1}
For every $\lambda_1 > 0$ there exists $\delta >0$ and a sequence of events $(G_n)_{n \in \N}$ such that, for every $\lambda \le \lambda_1$ and $n \in \N $,
\[ \P_\lambda[G_n] \ge 1-  e^{- \delta n} ,\]
and moreover on the event $G_n$, for every connected $\mathcal{O}' \in \mathfrak{O}$ with $|\mathcal{S}_{\mathcal{O}'}| \ge n$,
\begin{equation}
\label{e:vc1}
  \delta  |\mathcal{S}_{\mathcal{O}'}| \le  \Vol (\mathcal{O}')  \le     |\mathcal{S}_{\mathcal{O}'}|  . 
 \end{equation}
\end{proposition}
The non-trivial content of Proposition \ref{p:vc1} is the lower bound; roughly speaking the intuition is that since $\mathcal{O}'$ is connected it is unlikely to contain too many balls of very small radius, which means that $\mathcal{O}'$ must have volume comparable to the number of cubes $|\mathcal{S}_{\mathcal{O}'}|$ that it intersects. Note that if the radius distribution $\mu$ were bounded away from zero then \eqref{e:vc1} would hold even without the event $G_n$, but we do not wish to assume this.

\smallskip
Before proceeding we deduce two corollaries which will be used in the next section. The first states that the density $\theta(\lambda)$ could equivalently be defined as $\theta(\lambda) = \P_\lambda[\Vol(\C) = \infty]$, and the second shows that the probabilities of the cluster exceedence and magnetisation events $\{ \Vol(\C) \ge y\}$ and $\{0 \longleftrightarrow \mathcal{G}\}$ are continuous functions of $\lambda$.

\begin{corollary}
\label{c:vc}
For every $\lambda > 0$, almost surely
\[ \{0 \longleftrightarrow \infty \} = \{ |\mathcal{S}_\C| = \infty \} =  \{\Vol(\C) = \infty \} . \]
Moreover, for every $y, \rho > 0$ the maps
\[ \lambda \mapsto \P_\lambda[ \Vol(\C) \ge y] \quad \text{and} \quad \lambda \mapsto \P_{\lambda, \rho}[0 \longleftrightarrow \mathcal{G}] \]
are continuous functions of $\R_+$.
\end{corollary}

\begin{proof}
 For subsets $X,Y,Z \subset \R^d$ define the connectivity relation
\[  \{X \stackrel{Z}{\longleftrightarrow} Y \} = \{ \text{there exists a path in $Z$ between $X$ and $Y$} \} .\]
For $R > 0$ define
\[ \mathcal{O}_R  =  \bigcup_{\substack{(x, r) \in \eta \text{ such that}\\ \{x + B_r\} \cap B_R  \neq \emptyset}} \{ x + B_r \}  \quad \text{and} \quad \mathcal{C}_R = \{ x \in \R^d :  0 \stackrel{\mathcal{O}_R}{\longleftrightarrow} x \} , \]
which are, respectively, the union of all balls in the Poisson-Boolean percolation which intersect $B_R$ and the cluster of the origin in this set. For later use we note that, due to the integrability \eqref{a:wk}, $\mathcal{O}_R$ is generated by a Poisson point process of finite total intensity, and hence $\lambda \mapsto \P_{\lambda}[A]$ is continuous for any event $A$ measurable with respect to $\mathcal{O}_R$.
 
Let $\lambda_1 > 0$ be given, and let $\delta >0$ and $(G_n)_{n \in \N}$ be defined as in the statement of Proposition~\ref{p:vc1}. Notice that the event $\{0 \stackrel{\C_R}{\longleftrightarrow} \partial B_R\}$ implies that $|\mathcal{S}_{\C_R}| \ge \lceil R/2 \rceil$. Applying Proposition \ref{p:vc1} to $\C_R \in \mathfrak{O}$ we deduce that, for all $\lambda \le \lambda_1$ and $R > 0$,
 \begin{equation}
     \label{e:vcp1}
  \P_\lambda[ \{0 \stackrel{\C_R}{\longleftrightarrow} \partial B_R\}   \setminus \{ \Vol(\C_R) \ge \delta  R / 2  \} ] \le 1 - \P_\lambda[G_{\lceil R/2 \rceil}] \le e^{-\delta R / 2} .   \end{equation}
  We will show that all statements in Corollary \ref{c:vc} follow from \eqref{e:vcp1}.

For the first statement, by \eqref{e:vcp1} and the Borel-Cantelli lemma there almost surely exists an $R_0 > 0$ such that, for all natural numbers $R \ge R_0$,
\[  \{0 \stackrel{\C_R}{\longleftrightarrow} \partial B_R\}   \subset  \{ \Vol(\C_R) \ge \delta  R / 2  \}  . \]
 Hence almost surely
\[  \{0 \longleftrightarrow \infty \} = \cap_{R \ge R_0}   \{0 \stackrel{\C_R}{\longleftrightarrow} \partial B_R\}   \subset \cap_{R \ge R_0 }  \{ \Vol(\C_R) \ge \delta   R / 2  \}= \{ \Vol(\C) = \infty \} ,\] 
 which completes the proof since trivially $\{ \Vol(\C) = \infty \} \subset \{ |\mathcal{S}_\C| = \infty \} = \{0 \longleftrightarrow \infty \}$.

For the second statement, we proceed by local approximation of the events $\{ \Vol(\C) \ge y \}$ and $\{ 0 \longleftrightarrow \mathcal{G}\}$. Let $y , \rho> 0$ be given. We claim that, by \eqref{e:vcp1}, there exists a $c = c(\lambda_1, y, \rho) > 0$ such that, for every $\lambda \le \lambda_1$ and $R \ge 1$,
 \begin{equation}
     \label{e:cor1}
\P_{\lambda}[\{0 \stackrel{\C_R}{\longleftrightarrow} \partial B_R\} \setminus \{Vol(\C_R) \ge y  \} ] \le  e^{- c R}      
 \end{equation} 
 and
 \begin{equation}
     \label{e:cor2}
     \P_{\lambda,\rho}[  \{0 \stackrel{\C_R}{\longleftrightarrow} \partial B_R\} \setminus \{0 \stackrel{\C_R}{\longleftrightarrow} \mathcal{G}\}   ] \le  e^{-c  R}     . 
 \end{equation}
 Indeed \eqref{e:cor1} follows directly from \eqref{e:vcp1} by taking $R$ sufficiently large such that $\delta  R / 2  \ge y$ and adjusting constants. For \eqref{e:cor2}, we have by \eqref{e:vcp1}
 \[  \P_{\lambda,\rho}[ \{0 \stackrel{\C_R}{\longleftrightarrow} \partial B_R\} \setminus  \{0 \stackrel{\C_R}{\longleftrightarrow} \mathcal{G}\}   ] \le \P_{\lambda,\rho}[ \{ \Vol(\C_R) \ge \delta  R / 2  \} \setminus  \{0 \stackrel{\C_R}{\longleftrightarrow} \mathcal{G}\}    ] + e^{-\delta R / 2} .\]
 On the other hand, by conditioning on $\mathcal{O}$ and using the independence of the ghost field $\mathcal{G}$,
\begin{align*}
\P_{\lambda,\rho}[ \{ \Vol(\C_R) \ge  \delta  R / 2  \} \setminus \{0 \stackrel{\C_R}{\longleftrightarrow} \mathcal{G}\} ]  & = \E_\lambda \big[    \id_{\Vol(\C_R) \ge  \delta R / 2}    \E_{\lambda,\rho}[  \{ |\C_R \cap \mathcal{G}| = 0  \}  \, | \, \mathcal{O} ]    \big]   \\
& \le  \E_\lambda \big[    \id_{\Vol(\C_R) \ge \delta R / 2}   e^{-\rho \Vol(\C_R) }   \big]   \le e^{-\rho \delta R / 2}  ,
\end{align*}
which gives the result.

Given \eqref{e:cor1}, the continuity of $\lambda \mapsto \P_\lambda[ \Vol(C) \ge y]$ follows by noticing that $\{\Vol(\C_R) \ge y\} \subset \{\Vol(\C) \ge y\}$ and
\[ \{\Vol(\C) \ge y\} \setminus \{\Vol(\C_R) \ge y\}  \subset \{0 \stackrel{\C_R}{\longleftrightarrow} \partial B_R\}   \setminus \{\Vol(\C_R) \ge y\}.  \]
Together with \eqref{e:cor1} we deduce that, for $\lambda \le \lambda_1$,
\[ \P_{\lambda}[\Vol(\C) \ge y] \ge \P_{\lambda}[\Vol(\C_R) \ge y ] \ge  \P_{\lambda}[\Vol(\C) \ge y] -  e^{-c  R} , \]
and since $\lambda \mapsto \P_\lambda[ \Vol(\C_R) \ge y]$ is continuous (recall that the event $\{\Vol(\C_R) \ge y\}$ is measurable with respect to $\mathcal{O}_R$), the desired continuity follows from the uniform limit theorem by taking $R \to \infty$. 

Similarly, since $\{0  \stackrel{\C_R}{\longleftrightarrow} \mathcal{G} \} \subset \{0 \longleftrightarrow \mathcal{G}\}$ and
\[ \{0 \longleftrightarrow \mathcal{G}\} \setminus \{0  \stackrel{\C_R}{\longleftrightarrow} \mathcal{G} \}  \subset  \{0 \stackrel{\C_R}{\longleftrightarrow} \partial B_R\}    \setminus \{0  \stackrel{\C_R}{\longleftrightarrow} \mathcal{G} \}, \]
by \eqref{e:cor2} we deduce that
\[ \P_{\lambda,\rho}[0 \longleftrightarrow \mathcal{G}] \ge \P_{\lambda,\rho}[0  \stackrel{\C_R}{\longleftrightarrow} \mathcal{G}] \ge  \P_{\lambda,\rho}[0 \longleftrightarrow \mathcal{G}] -  e^{-c  R}  \]
and we again deduce the continuity of $\lambda \mapsto \P_{\lambda, \rho}[0 \longleftrightarrow \mathcal{G}]$ by taking $R \to \infty$.
\end{proof}

We move on to the proof of Proposition \ref{p:vc1}. Since the upper bound is trivial, it remains to prove the lower bound. 

\smallskip
A collection of cubes $\mathcal{S} \subset S$ is \textit{connected} if its index set in $\Z^d$ is connected (i.e.\ connectivity is induced by cubes sharing a $(d-1)$-dimensional face rather than intersection); a connected $\mathcal{S}$ is an \textit{animal} if it contains $S_0 = [0,1]^d$. For a cube $S_i$ denote by $S_i^+$ its union with its neighbours.

\smallskip 
The events $(G_n)_{n \in \N}$ in Proposition \ref{p:vc1} will be defined by introducing a notion of `good' cube. For $\varepsilon > 0$ and a cube $S_i \in S$, define $G_\varepsilon(S_i)$ to be the event that $S_i$ does not intersect any ball in $\mathcal{O}$ with radius less than $\varepsilon$, i.e.\ 
\[ G_\varepsilon(S_i) = \Big\{ S_i \cap \Big( \cup_{(x, r) \in \eta, r \le \varepsilon} \{x + B_r \} \Big) = \emptyset \Big\}. \]
For $\varepsilon, \delta>0$ and a finite collection $\mathcal{S} = (S_{i_j})_{j = 1,\ldots, k}$ of cubes, let $G_{\varepsilon,\delta}(\mathcal{S})$ denote the event that at least $\delta |\mathcal{S}|$ of the cubes in $\mathcal{S}$ verify $G_\varepsilon(S_{i_j})$. Finally, for $n\in\N$ define the event $G_{n,\varepsilon,\delta}$ that $G_{\varepsilon,\delta}(\mathcal{S})$ is satisfied for every animal $\mathcal{S}$ with $n \le |\mathcal{S}| < \infty$.

\begin{lemma}
\label{l:comparison1}
Fix $\varepsilon \in (0,1/2)$, $\delta > 0$, and $n \in \N$, and let $c_d > 0$ denote the volume of the unit ball in $\R^d$. On the event $G_{n, \varepsilon, \delta}$, for every connected $\mathcal{O}' \in \mathfrak{O}$ with $|\mathcal{S}_{\mathcal{O}'}| \ge n$,
\begin{equation}
\label{e:comp1}  \Vol( \mathcal{O}' )  \ge c_d 3^{-d} \varepsilon^d \delta  |\mathcal{S}_{ \mathcal{O}'} |  .
\end{equation}
If $ |\mathcal{S}_{ \mathcal{O}' }  | = \infty$, \eqref{e:comp1} is interpreted as $  \Vol( \mathcal{O}' )   = \infty$.
\end{lemma}

\begin{proof}
First consider the case that $n \le |\mathcal{S}_{\mathcal{O}'}| < \infty$. We use the simple geometric fact that if $\varepsilon \in (0,1/2)$ then any ball of radius larger than $\varepsilon$ that intersects a cube $S_i$ has volume at least $ c_d\varepsilon^d$ contained in $S^+_i$ (see Figure \ref{f:neigh}). 
\begin{figure}
\centering
\includegraphics[height=5cm]{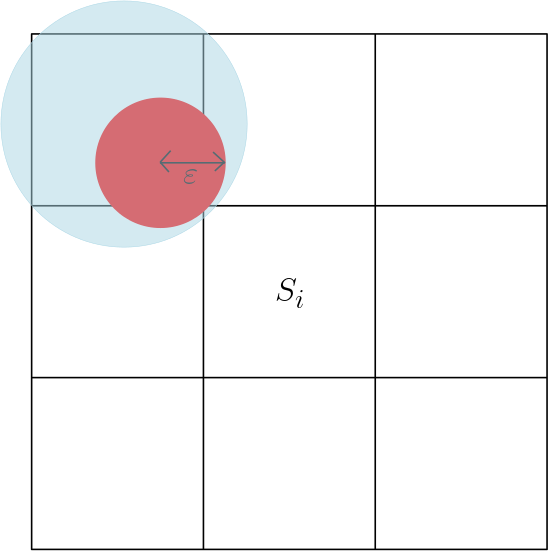}
\caption{An illustration of the geometric fact used in the proof of Lemma~\ref{l:comparison1}. Every ball of radius $\varepsilon < 1/2$ that intersects $S_i$ (in red) is contained in $S^+_i$, and any ball of larger radius intersecting $S_i$ (in light blue) contains such a ball.}
\label{f:neigh}
\end{figure}
Now since $\mathcal{O}'$ is connected, $\mathcal{S}_{\mathcal{O}'}$ is an animal, so the event $G_{n,\varepsilon,\delta}$ implies that there exist at least $\delta  |\mathcal{S}_{\mathcal{O}'} |$ cubes in $\mathcal{S}_{\mathcal{O}'}$ which only intersect balls in $\mathcal{O}$ of radius larger than $\varepsilon$. For each such cube $S_i$, using the geometric fact mentioned above we have
\[  \Vol( S_i^+ \cap \mathcal{O}' ) \ge c_d\varepsilon^d. \]
Summing over $S_i \in \mathcal{S}_{\mathcal{O}'}$, and since each cube of $\mathcal{S}_{\mathcal{O}'}$ contributes to the sum at most $3^d$ times, we have 
\[ 3^d \Vol( \mathcal{O}' ) \ge \delta |\mathcal{S}_{\mathcal{O}'}| c_d \varepsilon^d  . \]
as required. On the other hand, if $|\mathcal{S}_{\mathcal{O}'}| = \infty$ then for any $n' \in \N$ we can find a subset $\mathcal{O}'' \subset \mathcal{O}'$ such that $\mathcal{O}'' \in \mathfrak{O}$ and $\mathcal{S}_{\mathcal{O}''}$ is an animal with $|\mathcal{S}_{\mathcal{O}''}| \ge n'$. Taking $n' \to \infty$ and using the same argument as before we deduce that 
\[ \Vol( \mathcal{O}' )  \ge  \Vol(  \mathcal{O}''  )  \ge  c_d 3^{-d} \varepsilon^d \delta n' \to  \infty  \]
as required.
\end{proof}

We next argue that, for small enough $\varepsilon,\delta > 0$, the event $G_{n,\varepsilon,\delta}$ occurs with overwhelming probability:

\begin{lemma}
\label{l:comparison2}
For every $\lambda_1>0$ there exists a $\delta_0 >0$, such that for any $\lambda \le \lambda_1$ and $n \in \N$,
\[  \P_\lambda[G_{n,\delta_0,\delta_0}] \ge 1-  e^{-  \delta_0 n}. \]
\end{lemma}

The proof will make use of a result of Lee on lattice animals. Abusing notation slightly, we will call a subset $\xi \subset \mathbb{Z}^d$ an \textit{animal} if it is connected in $\Z^d$ and contains the origin.

\begin{proposition}[{\cite[Theorem 5]{lee93}}]
\label{lee}
Let $(Y_v)_{v\in\Z^d}$ be i.i.d non-negative random variables with $\P[Y_0=0]<p_c(d)$, where $p_c(d) > 0$ is the critical parameter of Bernoulli site percolation on~$\Z^d$. Then there exist $\delta_0,c>0$ such that, for every $n\in\N$,
\[ \P\Big[\text{there exists an animal $\xi \subset \Z^d$ such that $|\xi|\geq n$ and $\sum\limits_{v\in\xi} Y_v
\leq \delta_0 n$}\Big]  \le 4 e^{-cn}.  \]
\end{proposition}

\begin{proof}[Proof of Lemma \ref{l:comparison2}]
For $i \in\Z^d$ and $\varepsilon>0$ define the random variable $Y^\varepsilon_i$ to be the indicator of the event that $S^+_i$ does not intersect $\{(x,r) \in \eta : r \le \varepsilon\}$. Now for every $\lambda > 0$, $i \in \Z^d$, and $\varepsilon > 0$, we have
\[   \P_\lambda [Y^\varepsilon_i=0] = \P[ \textrm{Pois}( 3^d \lambda  \mu( [0,\varepsilon] ) ) \ge  1 ]  =  1 - e^{ -  3^d \lambda  \mu( [0,\varepsilon] ) } , \]
and so for $\varepsilon  = \varepsilon(\lambda_1)> 0$ sufficiently small,
\begin{equation}
\label{smallparam}
\P_\lambda[Y^{\varepsilon}_i=0] <p_c(d) 
\end{equation}
for every $\lambda \le \lambda_1$ and $i \in \Z^d$. 

We cannot directly apply Proposition \ref{lee} to the random variables $(Y^{\varepsilon}_i)_{i \in \Z^d}$ since they are not independent. Nevertheless, they are finite-range dependent since $Y^\varepsilon_i$ and $Y^\varepsilon_j$ are independent unless $S_i$ and $S_j$ have a common neighbour, and so by a classical result of Liggett, Schonmann and Stacey \cite[Theorem 0.0]{lss97}, for sufficiently small $\varepsilon = \varepsilon(\lambda_1) > 0$ there exists a family $(Z_i)_{i\in\Z^d}$ of i.i.d.\ Bernoulli random variables which is stochastically dominated by $(Y^\varepsilon_i)_{i \in \Z^d}$ for every $\lambda \le \lambda_1$ and also verifies $\P[Z_i=0] <p_c(d)$ for every $i \in \Z^d$.

Applying Proposition \ref{lee} to $(Z_i)_{i \in \Z^d}$ yields constants $\delta_0,c>0$, depending only on $\mu$ and $\varepsilon$, such that, for every $\lambda \le \lambda_1$ and $n \in \N$,
\begin{align*}
\begin{split}
&\P_\lambda\Big[\text{there exists an animal $\mathcal{S} = (S_i)$ with $n \le |\mathcal{S}| < \infty$ such that}
\\
&\quad \qquad \qquad \text{at most $\delta_0 n$ of the cubes in $\mathcal{S}$ verify $\{Z_i=1\}$}\Big] \leq 4e^{-cn}.
\end{split}
\end{align*}
By stochastic domination, the same is true for $Y_i^\varepsilon$ replacing $Z_i$. Since for $\varepsilon < 1/2$ we have
\[ Y^\varepsilon_i  = 1 \quad \Longrightarrow \quad G_{\varepsilon}(S_i) \text{ holds} , \]
we deduce that
\begin{align*}
& \P_\lambda\Big[\text{for every animal $\mathcal{S} = (S_i)$ with $n \le |\mathcal{S}| < \infty$ at least $\delta_0 n$ of the cubes verify $G_{\varepsilon}(S_i)$}\Big] \\
& \qquad \ge 1- 4e^{-cn} . 
\end{align*}
Since 
\[  \bigcap_{m \ge n}  \Big\{ \text{for every animal $\mathcal{S} = (S_i)$ with $ |\mathcal{S}| = m $ at least $\delta_0 m$ of the cubes verify $G_{\varepsilon}(S_i)$} \Big\} \]
is contained in $G_{n, \varepsilon, \delta_0}$, this implies that, for every $\lambda \le \lambda_1$ and $n\in\N$, 
\[ \P_\lambda[ G_{n, \varepsilon, \delta_0} ]   \ge 1 - 4\sum\limits_{m\geq n}e^{-cm}=1 - 4e^{-cn}\frac{1}{1-e^{-c}} \]
and we deduce the result by adjusting constants.
\end{proof}

\begin{proof}[Proof of Proposition \ref{p:vc1}]
Let $\lambda_1 > 0$ be given and let $\delta_0 > 0$ be the constant appearing in Lemma \ref{l:comparison2}. Then by combining Lemmas \ref{l:comparison1} and \ref{l:comparison2} the result follows for the constant $\delta =  \min\{ c_d 3^{-d} \delta_0^{d+1}, \delta_0\}$ and the events $G_n = G_{n, \delta_0, \delta_0}$. 
\end{proof}

\subsection{From the number of intersecting cubes to the Poisson-Boolean volume}

The next lemma shows that, up to multiplicative constants, we may neglect the extra dimension of Poisson-Boolean space.

\smallskip
For a point $x \in \R^d$, we define the \emph{cone above $x$} to be the set defined by
\[\cone(x) =  \bigcup\limits_{y\geq 0}(B_y+x,y) = \{(x',y)\in\R^d\times\R_+ : \ \dist( x',x) \le y\}     \subseteq\R^d\times\R_+  , \]
and for $D \subset \R^d$, define $\cone(D) = \bigcup\limits_{x\in D}\cone(x)$ to be the \textit{cone above $D$} (see Figure \ref{f:cones}).
\begin{figure}
\centering
\includegraphics[height=6cm]{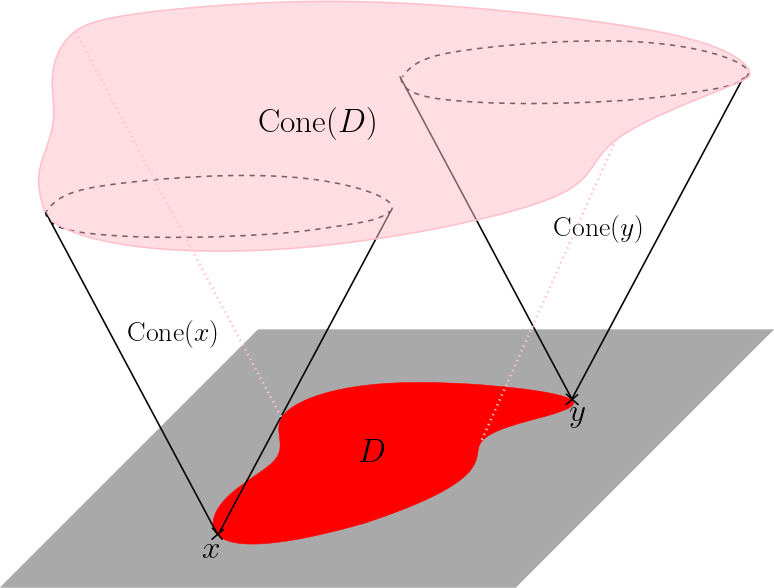}
\caption{An illustration of the cone above a compact set $D$ in the case $d=2$.}
\label{f:cones}
\end{figure}

\begin{lemma}
\label{l:vc2}
There exists a constant $c_\mu > 0$ such that, for every finite collection of cubes $\mathcal{S} = (S_{i_j})_{j = 1,\ldots, k}$,
\[  |\mathcal{S}| \le \PVol( \cone( \cup_j   S_{i_j} ) ) \le c_\mu |\mathcal{S}| . \]
\end{lemma}

\begin{proof}
Fix a collection $\mathcal{S} = (S_{i_j})_{j = 1,\ldots, k}$. For the lower bound,
\[\PVol( \cone( \cup_j   S_{i_j} ) )\ge \PVol(\cup_j   S_{i_j} \times (0,\infty))=|\mathcal{S}|, \]
where we used that $\mu$ is a probability measure. For the upper bound, we write
\[\PVol( \cone( \cup_j   S_{i_j} ) )=\PVol( \cup_j   \cone(S_{i_j} ) )\le \sum_j \PVol( \cone(S_{i_j} ) )=c_\mu |\mathcal{S}|,
\]
where $c_\mu$ is the Poisson-Boolean volume of the cone above a single cube. Both arguments are illustrated in Figure \ref{f:pyramids}.
\end{proof}

\begin{figure}
\centering
\includegraphics[height=3.4cm]{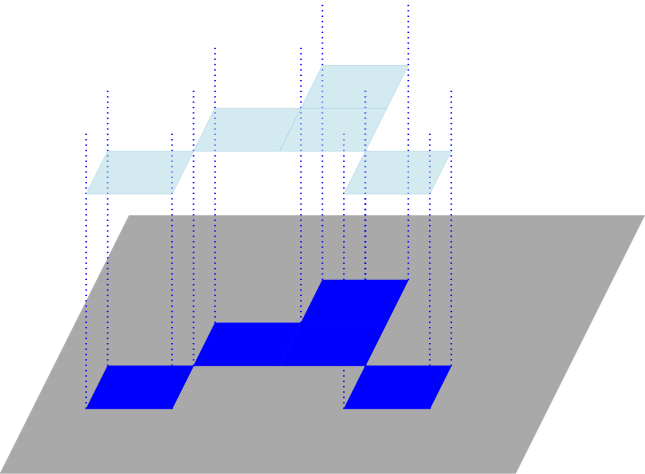}
\includegraphics[height=2.1cm]{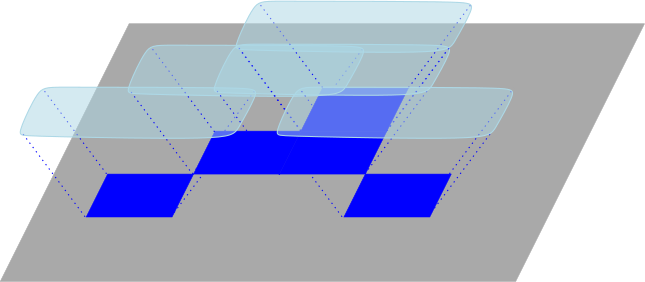}
\caption{The two bounds in the proof of Lemma~\ref{l:vc2} in the case $d=2$: the Poisson-Boolean volume of the union of cones is at least (resp.\ most) that of the shape depicted on the left (resp.\ right), $\mathcal{S}$ being represented in dark blue.}\label{f:pyramids}
\end{figure}

\medskip

\section{Proof of the main results}
\label{s:mr}

In this section we complete the proof of Theorems \ref{t:s}--\ref{t:m}, which will consist of applying the entropic bounds from Section \ref{s:entropic} to well-chosen algorithms.

\subsection{The template algorithm}
We wish to define algorithms which verify the events (recall that $\mathcal{G}$ is the ghost field introduced in Section \ref{s:entropic})
\begin{equation}
\label{e:a}
\{0\longleftrightarrow \partial B_R\} , \quad  \{\Vol(\C) \ge y \}  \quad \text{and} \quad    \{ 0 \longleftrightarrow \mathcal{G} \}    ,
 \end{equation}
and it will be convenient to base these on a single template. 

\smallskip
Recall that an algorithm is an adapted procedure which sequentially reveals $\eta|_{C_{i_j}}$ for a random sequence of hypercubes $(C_{i_j})_{j \in \N}$. For such an algorithm, define $\eta_n =   \cup_{j \le n} \eta|_{C_{i_j}}$ to be the points that are revealed by the algorithm up to step $n$, and similarly define the revealed portion of the Poisson-Boolean model and the cluster of the origin in this set
\[ \mathcal{O}_n  = \bigcup\limits_{(x, r) \in \eta_n} \{x + B_r \}   \quad \text{and} \quad  \mathcal{C}_n = \{ x \in \R^d : 0 \stackrel{\mathcal{O}_n}{\longleftrightarrow} x  \} . \]

\begin{algorithm*}
\label{algo}
$\,$
\begin{itemize}
\item Initialise $n=0$ and the revealed set $\mathcal{R} = \emptyset$, and define the active set $\mathcal{A}$ to be the collection of hypercubes which intersect $\cone(0)$, the cone above $0$.
\item Iterate the following:
\begin{itemize}
\item Increment $n \to n+1$.
\item If $\mathcal{A}$ is empty skip to the end of the iteration (and thus enter an infinite loop).
\item Assuming $\mathcal{A}$ is non-empty, define $C_{i_n}$ to be the hypercube in $\mathcal{A}$ whose index $i \in \Z^d \times \Z^+$ has smallest sup-norm, breaking ties arbitrarily.
\item Add $C_{i_n}$ to the revealed set $\mathcal{R}$ and update the active set
\[ \mathcal{A} = \Big\{ C_i : C_i \cap  \cone(\mathcal{S}_{\C_n}) \neq \emptyset \Big\}  \setminus \mathcal{R}  .\]
In words, the active set consists of all unrevealed hypercubes which are in the cone above the cubes that intersect the cluster of the origin in the revealed portion of the model.
\item If $\{ \mathcal{O}_n \in A \}$ holds, terminate and return value $1$.
\end{itemize}
\end{itemize}
\end{algorithm*}

See Figure \ref{f:algo} below for an illustration of $\mathcal{T}$ for the event $A = \{0 \longleftrightarrow \partial B_R\}$.

\smallskip
We also define a \emph{magnetic} version of $\mathcal{T}$ in the setting where we have a ghost field $\mathcal{G}$ and the event $A$ may depend on $\mathcal{G}$; in this case, $\mathcal{T}$ first reveals $\mathcal{G}$ before proceeding as before, and terminates if $\{ (\mathcal{O}_n, \mathcal{G}) \in A \}$.

\smallskip
For the entropic bounds in Section \ref{s:entropic} to apply, we need conditions guaranteeing that $\mathcal{T}$ `locally determines' the event $A$. We say that $A$ is \textit{increasing} if $\{ \mathcal{O} \in A\} \subset \{\mathcal{O}' \in A\}$ for any $\mathcal{O} \subset \mathcal{O}'$. We say that $A$ is a \textit{local cluster event} if it is almost-surely witnessed by $\{ x \in D : 0 \stackrel{\mathcal{O}\cap{D}}{\longleftrightarrow} x \}$ for some (random) compact $D \subset \R^d$ (or witnessed by $(\{ x \in D : 0 \stackrel{\mathcal{O}\cap{D}}{\longleftrightarrow} x \}, \mathcal{G})$ in the magnetic case); clearly the events in \eqref{e:a} are local cluster events. 

\begin{lemma}
Suppose $A$ is an increasing local cluster event. Then $\mathcal{T}$ locally determines $A$.
\end{lemma}
\begin{proof}
We need to verify the three conditions of the definition of `locally determine'. By definition, the hypercube $C_{i_n}$ that $\mathcal{T}$ reveals at step $n$ is in $\cone( \mathcal{S}_{\C_n} ) \subset \cone(  \{ S_j : \|j\|_\infty \le n \} )$. Also, since $C_{i_n}$ is chosen among the active hypercubes with lowest sup-norm, and the hypercubes above $S_0$ are all initially active, the $(d+1)$-coordinate of $C_{i_n}$ is at most $n$. Hence $C_{i_n}  \in \{ C_j : \|j\|_\infty \le 2n\}$, and the first condition holds. 

Next observe that, due to the integrability \eqref{a:wk}, the cone $\cone(D)$ above a compact $D \subset \R^d$ has finite Poisson-Boolean volume, and so almost surely $\eta|_{\cone(D)}$ is finite. Further, define 
\[ \mathcal{S}' = \{ C_i \text{ intersecting } \cone( \{ x \in D : 0 \stackrel{\mathcal{O}\cap D}{\longleftrightarrow} x \}) : C_i \cap \eta \neq \emptyset\} , \]
 and let $M$ be the maximal sup-norm of the indices in $\mathcal{S}'$. By definition, $\mathcal{T}$ reveals at most $(2M+1)^{d+1}$ hypercubes before revealing all the hypercubes in $\mathcal{S}'$ (since if $\mathcal{S}'$ is not yet revealed and $\mathcal{T}$ is yet to terminate then there is at least one hypercube in $\mathcal{S}'$ which is active), and if all the hypercubes in $\mathcal{S}'$ are revealed, the whole of $ \{ x \in D : 0 \stackrel{ \mathcal{O}\cap D}{\longleftrightarrow} x \}$ is known. Hence, for every compact $D \subset \R^d$, $\mathcal{T}$ eventually reveals the set $ \{ x \in D : 0 \stackrel{\mathcal{O}\cap D}{\longleftrightarrow} x \}$ unless it terminates before doing so. We have assumed that $A$ is a local cluster event, thus we consider $D$ to be its witness compact set. The previous argument justifies that if $A$ occurs then it will eventually be witnessed by $\mathcal{O}_n$, for $n$ smaller or equal to $(2M+1)^{d+1}$. At that point $\mathcal{T}$ will terminate and return $1$. On the other hand, if $A$ does not occur then, since $A$ is increasing, $\{\mathcal{O}_n \in A\}$ never holds and $\mathcal{T}$ does not terminate by definition. Together these verify the second and third conditions.
\end{proof}

\subsection{Bounds on the volume of the revealed set}
Recall that the entropic bounds in Section~\ref{s:entropic} are stated in terms of the Poisson-Boolean volume of the revealed set. The next proposition bounds this in terms of the ambient volume:

\begin{proposition}
\label{p:vc3}
For every $0 < \lambda_0 < \lambda_1$ there exists a $c > 0$ such that, for every $\lambda \in [\lambda_0, \lambda_1]$, $\rho > 0$, and event $A$, the algorithm $\mathcal{T}$ satisfies
\[ \E_{\lambda, \rho}[  \PVol( \mathcal{W}_\mathcal{T})  ] \le c \E_{\lambda}[ \Vol(\mathcal{C})]   ,\]
where we recall that $\mathcal{W}_\mathcal{T}$ is the union of the hypercubes revealed by $\mathcal{T}$.
Moreover, if $y \ge 1$ and $A = \{ \Vol(\C) \ge y\}$, then
\[ \E_\lambda[  \PVol( \mathcal{W}_\mathcal{T})  ] \le  c \int_0^y \P_\lambda[\Vol(\C) \ge u] du   . \]
Finally, for every $0 < \lambda_0 < \lambda_1$ and $\rho_0 > 0$ there exists a $c > 0$ such that, for every $\lambda \in [\lambda_0, \lambda_1]$ and $\rho \in (0,\rho_0]$, if $A = \{ 0 \longleftrightarrow \mathcal{G} \}$, then
\[  \E_{\lambda,\rho}[  \PVol( \mathcal{W}_\mathcal{T})  ] \le  c \rho^{-1}  \E_\lambda\big[ 1 - e^{-\rho \Vol(\mathcal{C})} \big]  .\] 

\end{proposition}
\begin{proof}
For the first statement we separate the cases (i) $\PVol(\mathcal{W}_\mathcal{T})$ is infinite with positive probability, and (ii) $\PVol(\mathcal{W}_\mathcal{T})$ is finite almost surely.

In the first case it suffices to show that $ \{ \PVol( \mathcal{W}_\mathcal{T}) = \infty \} \subset \{  \Vol(\mathcal{C}) = \infty \}$. Recall that $\mathcal{S}_{\C} \subset S$ denotes the collection of cubes which intersect $\C$, which is necessarily an animal. Since $\mathcal{T}$ always reveals, at step $n$, a hypercube in $\cone( \mathcal{S}_n ) \subset \cone(\mathcal{S}_\C) $, it follows that $\mathcal{W}_\mathcal{T} \subset \cone(\mathcal{S}_\C)$.
Hence by using the upper bound in Lemma \ref{l:vc2} and then Corollary \ref{c:vc}, for every $\lambda,\rho> 0$, almost surely
\begin{equation}
\label{e:infty}
 \{ \PVol( \mathcal{W}_\mathcal{T}) = \infty \} \subset \{ \PVol(\cone(\mathcal{S}_\C)) = \infty \} \subset \{ | \mathcal{S}_\C| = \infty \} = \{  \Vol(\mathcal{C}) = \infty \} 
 \end{equation}
 as required.
 
We turn to the second case. Fix the constant $ \delta > 0$ and the events $(G_n)_{n \ge \N}$ guaranteed to exist by Proposition~\ref{p:vc1} (applied to $\lambda_1$), and also the constant $c_\mu > 0$ defined by Lemma \ref{l:vc2}. For $s > 0$ define $n = n_s = \lceil s /c_\mu \rceil \in \N$. Recalling that $\mathcal{W}_\mathcal{T} \subset \cone(\mathcal{S}_\C)$ and then using the upper bound in Lemma \ref{l:vc2} and the lower bound in Proposition~\ref{p:vc1} (applied to $\mathcal{C} \in \mathfrak{O})$, for every $\lambda \le \lambda_1$ and $s > 0$,
 \[ \{  \PVol(\mathcal{W}_\mathcal{T} ) \ge s \} \cap G_n \! \subset  \{ \PVol(\cone(\mathcal{S}_\C)) \ge s \} \cap G_n \! \subset  \{ |\mathcal{S}_\C| \ge s/c_\mu \} \cap G_n \! \subset \{ \Vol(\C) \ge \delta s / c_\mu  \} .\]
Hence we have
\[ \P_{\lambda,\rho}[\PVol(\mathcal{W}_\mathcal{T} ) \ge s ] \le \P_\lambda[ \Vol(\C) \ge \delta   s /c_\mu   ] +  e^{- \delta s /c_\mu  } . \]
Integrating over $s \in (0,\infty)$ (recall that we assume $ \PVol(\mathcal{W}_\mathcal{T})$ is finite almost surely), we have
\[ \E_{\lambda,\rho}[  \PVol( \mathcal{W}_\mathcal{T})  ] \le (c_\mu  / \delta)  \E_\lambda[ \Vol(\mathcal{C})]  + c_\mu/\delta .\]
Since $\lambda \mapsto \E_\lambda[ \Vol(\mathcal{C})]$ is increasing, and so bounded away from zero over $\lambda \ge \lambda_0$, this implies the first statement by adjusting constants.

For the second statement, recall that $\tau$ denotes the (possibly infinite) stopping time of $\mathcal{T}$, and define $\bar{\tau} = \tau - \id_{\tau < \infty}$. The crucial observation is that $\Vol( \C_{\bar{\tau}} ) < y $ since otherwise $\mathcal{T}$ would have verified the event $A = \{ \Vol(\C) \ge y\}$ prior to step $\bar{\tau}$ and terminated. On the other hand $\mathcal{W}_\mathcal{T} \subset \cone(\mathcal{S}^+_{\C_{\bar{\tau}}})$, where $\mathcal{S}^+_{\C_{\bar{\tau}}} = \bigcup\limits_{ S_i \in \mathcal{S}_{\C_{\bar{\tau}}}}S_i^+$ is the union of $\mathcal{S}_{\C_{\bar{\tau}}}$ with its neighbours. Noticing that $|\mathcal{S}^+_{\C_{\bar{\tau}}}| \le 3^d |\mathcal{S}_{\C_{\bar{\tau}}}|$, redefining $n = n_s = 3^{-d} s / c_\mu $, and using similar reasoning as before (this time applying Proposition \ref{p:vc1} to $\C_{\bar{\tau}}$), for every $\lambda \le \lambda_1$ and $s > 0$,
 \begin{align*}
  \{  \PVol(\mathcal{W}_\mathcal{T} ) \ge s \} \cap G_{n} &  \subset  \{ \PVol( \cone(\mathcal{S}^+_{\C_{\bar{\tau}}})) \ge s \} \cap G_{n} \subset  \{ |\mathcal{S}^+_{\C_{\bar{\tau}}}| \ge s/c_\mu  \} \cap G_{n}  \\
   & \subset  \{ |\mathcal{S}_{\C_{\bar{\tau}}}| \ge 3^{-d} s/c_\mu  \} \cap G_{n}   \subset \{ \Vol( \C_{\bar{\tau}} ) \ge \delta 3^{-d}  s / c_\mu     \}  .
  \end{align*}
  Since $\Vol( \C_{\bar{\tau}} ) < y $, the latter event is empty if $s \ge c_\mu 3^d y/\delta $. Hence 
\[ \P_\lambda[\PVol(\mathcal{W}_\mathcal{T} ) \ge s  ]  \le   \begin{cases}   \P_\lambda[ \Vol(\C_{\bar{\tau}}) \ge \delta   3^{-d}  s /c_\mu    ] +  e^{- \delta  3^{-d} s /c_\mu   } & \text{if } s \le c_\mu 3^d y/\delta   , \\ e^{- \delta  3^{-d} s /c_\mu   } & \text{if } s \ge c_\mu 3^d y/\delta   .  \end{cases} \]
Using that $\Vol(\C) \ge  \Vol(\C_{\bar{\tau}})$, and integrating over $s \in (0, \infty)$, gives
\[ \E_\lambda[\PVol(\mathcal{W}_\mathcal{T} ) ] \le (c_\mu 3^d / \delta) \int_0^y \P_\lambda[ \Vol(\C)  \ge u ] du  +  c_\mu 3^d/\delta    , \]
and since $ \int_0^y \P_\lambda[ \Vol(\C)  \ge u ] du$ is bounded away from zero over $\lambda \ge \lambda_0$ and $y \ge 1$, the result follows by adjusting constants.

For the third statement. Defining $\bar{\tau}$ as before, this time the crucial observation is that  $| \C_{\bar{\tau}} \cap  \mathcal{G} | = 0 $ since otherwise the algorithm would have verified the event $A = \{0 \longleftrightarrow \mathcal{G}\}$ prior to step $\bar{\tau}$ and terminated. Hence conditioning on $\mathcal{O}$ and using the independence of $\mathcal{G}$, for any $\lambda,\rho, s > 0$,
\[ \P_{\lambda,\rho}[ \Vol(\C_{\bar{\tau}}) \ge s ]  = \E \big[  \id_{ \Vol(\C) \ge s}  \P_{\lambda,\rho}[ \Vol(\C_{\bar{\tau}}) \ge s \: \big|\: \mathcal{O} ]  \big]  \le e^{-\rho s}  \P_{\lambda}[  \Vol(\C) \ge s ]  \]
 where we used that $\Vol(\C) \ge  \Vol(\C_{\bar{\tau}})$ in the first step. 
 To justify the second step, recall that $\P_{\lambda,\rho}[ \Vol(\C_{\bar{\tau}}) \ge s \: \big|\: \mathcal{O} ]$ is a measurable function of $\mathcal{O}$. And for any fixed configuration $O$ of $\mathcal{O}$ such that $\mathcal{C}$ has volume larger than $s$, $\P_{\lambda,\rho}[ \Vol(\C_{\bar{\tau}}) \ge s \: \big|\: \mathcal{O}=O]$ is the probability of there being no point of the ghost field in a deterministic set of volume larger than $ s$, thus the value of the function is smaller than $e^{-\rho s}$.
 Arguing as in the proof of the second statement we deduce that, for every $\lambda \le \lambda_1$ and $\rho,s > 0$,
\begin{align*}
  \P_{\lambda,\rho}[\PVol(\mathcal{W}_\mathcal{T} ) \ge s  ]  &   \le  \P_{\lambda,\rho}[ \Vol(\C_{\bar{\tau}}) \ge \delta 3^{-d} s  /c_\mu   ]  +  e^{- \delta  3^{-d} s /c_\mu   }   \\
 & \le  e^{-\rho \delta 3^{-d} s/c_\mu }  \P_{\lambda}[  \Vol(\C) \ge \delta 3^{-d} s/c_\mu  ] +  e^{- \delta 3^{-d} s /c_\mu    } .
  \end{align*}
 Integrating over $s \in (0, \infty)$ we obtain  
\begin{align*}
 \E_{\lambda,\rho}[\PVol(\mathcal{W}_\mathcal{T} ) ]  & \le (c_\mu 3^d  / \delta) \int_0^\infty e^{-\rho u }  \P_\lambda[ \Vol(\C)  \ge u ] du  + c_\mu 3^d /\delta \\
 & = (c_\mu 3^d / \delta) \rho^{-1} \E_\lambda[ 1 - e^{- \rho \Vol(\C) }    ]   + c_\mu 3^d /\delta ,
 \end{align*}
  where in the last step we used that, for every $\rho > 0$ and non-negative random variable $X$,
 \[   \int_0^\infty e^{-\rho u} \P[ X \ge u] du =  \E\Big[ \int_0^\infty  e^{-\rho u}  \id_{X \ge u} du \Big]   =   \E\Big[ \int_0^X  e^{-\rho u} du \Big] =   \rho^{-1} \E\big[ 1 - e^{-\rho X} \big]  .\]
Finally, since $ \rho^{-1} \E_\lambda[ 1 - e^{- \rho \Vol(\C) }    ]  $ is bounded away from zero over $\lambda \ge \lambda_0$ and $\rho \in (0, \rho_0]$, the result follows by adjusting constants.
\end{proof}
 
\subsection{Proof of Theorem \ref{t:s}}
It is enough to prove the result for $\lambda < \lambda_c$ sufficiently close to $\lambda_c$. We begin with \eqref{e:s1}. Fix $\lambda<\lambda_c$ and consider the algorithm $\mathcal{T}$ for the event $A = \{0 \longleftrightarrow \partial B_R\}$ (illustrated by Figure \ref{f:algo}). 
\begin{figure}
\centering
\includegraphics[width=10cm]{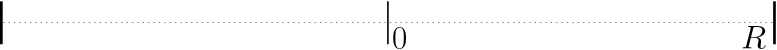}
\includegraphics[width=10cm]{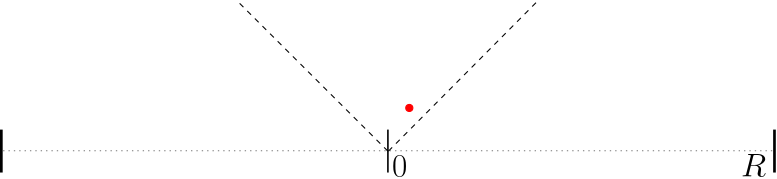}
\includegraphics[width=10cm]{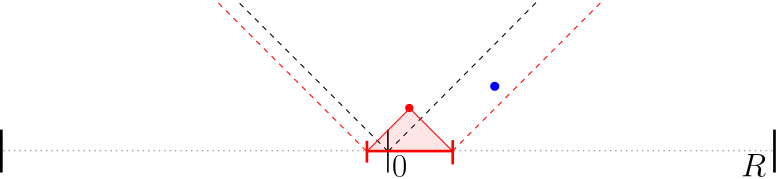}

\hspace{0.18cm}
\includegraphics[width=10.25cm]{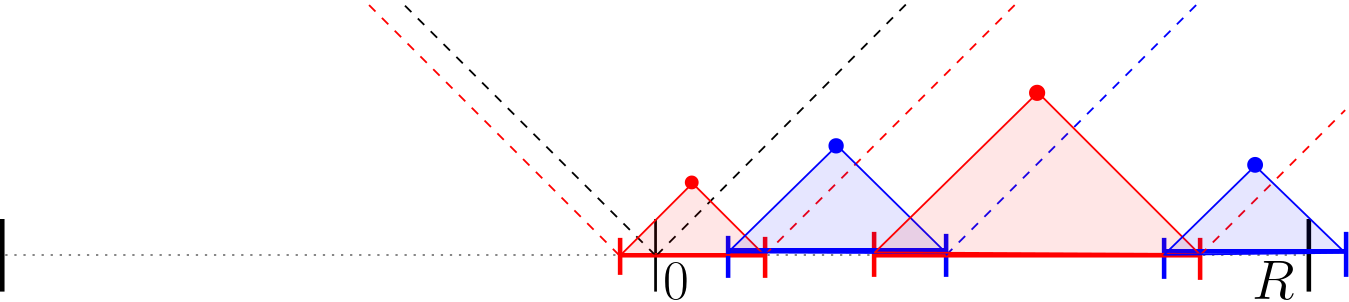}
\caption{Some early steps (i.e.\ the first two points discovered) and the final step of algorithm $\mathcal{T}$ in determining the event $ \{0 \longleftrightarrow \partial B_R\}$ in the case $d=1$. The red and blue points are points of $\eta$ corresponding to balls which contribute to the cluster $\C$. Their projection onto $\R$ is represented in matching colour. The dashed lines represent the boundaries of the successive cones that are active. The algorithm reveals the set of unit hypercubes intersecting these cones one by one starting with those closest to the origin.}
\label{f:algo}
\end{figure}
By the first statement of Proposition~\ref{p:entropic} (with $\lambda_1 = \lambda$ and $\lambda_2 = 2\lambda_c - \lambda > \lambda_c$),
\[ \P_{2\lambda_c-\lambda}[0\longleftrightarrow \partial B_R]-\P_{\lambda}[0\longleftrightarrow \partial B_R]  \le (2/\lambda_c)^{1/2}(\lambda_c-\lambda)  \sqrt{\P_{2\lambda_c-\lambda}[0\longleftrightarrow \partial B_R] \E_\lambda[ \PVol (\mathcal{W}_\mathcal{T}) ]} . \]
The first statement of Proposition \ref{p:vc3} gives that $\E_\lambda[  \PVol( \mathcal{W}_\mathcal{T})  ] \leq c_1 \E_\lambda[ \Vol(\mathcal{C})] $ for a constant $c_1 > 0$, for all $R$ and all $\lambda$ sufficiently close to $\lambda_c$. Thus we have
\[ \P_{2\lambda_c-\lambda}[0\longleftrightarrow \partial B_R]-\P_{\lambda}[0\longleftrightarrow \partial B_R]  \le (2/\lambda_c)^{1/2}(\lambda_c-\lambda)  \sqrt{c_1\P_{2\lambda_c-\lambda}[0\longleftrightarrow \partial B_R]\E_\lambda[ \Vol(\mathcal{C})] } . \]
Taking $R \to \infty$, this gives 
\[ \theta(2\lambda_c-\lambda)  \le  (2/\lambda_c)^{1/2}  (\lambda_c-\lambda)  \sqrt{ c_1\theta(2\lambda_c-\lambda) \E_\lambda[ \Vol(\mathcal{C})]}  \]
and hence
\[ \E_\lambda[ \Vol(\mathcal{C})] \ge c_1^{-1} (\lambda_c/2)  (\lambda-\lambda_c)^{-2} \theta(2\lambda_c-\lambda)   , \]
which is \eqref{e:s1}.

On to \eqref{e:s2}. Fix $\lambda < \lambda_c$, abbreviate $y =(\lambda_c-\lambda)^{-2}$, and consider the algorithm $\mathcal{T}$ for the event $A = \{ \Vol(\C) \ge y \}$. Similarly to as before, by applying Proposition \ref{p:entropic} (with $\lambda_1 = \lambda $ and $\lambda_2 = \lambda_c$) and then using Proposition~\ref{p:vc3}, we deduce that, for a constant $c_2 > 0$ and all $\lambda$ sufficiently close to $\lambda_c$,
\[ \P_{\lambda_c}[\Vol(\C) \ge y]-\P_{\lambda}[\Vol(\C) \ge y]  \le c_2 (\lambda_c-\lambda)  \sqrt{ \P_{\lambda_c}[\Vol(\C) \ge y] \chi(\lambda)} . \]
Applying Markov's inequality to $\P_{\lambda}[\Vol(\C) \ge y]$, this gives 
\begin{equation}
\label{e:s3}
 \P_{\lambda_c}[\Vol(\C) \ge y] - (\lambda_c-\lambda)^2\chi(\lambda) \le c_2 (\lambda_c-\lambda)  \sqrt{ \P_{\lambda_c}[\Vol(\C) \ge y] \chi(\lambda)} . 
 \end{equation}
There are two cases to consider: 
\begin{enumerate}
\item [ 1.] If $(\lambda_c-\lambda)^2\chi(\lambda)\geq \frac12\P_{\lambda_c}[\Vol(\C) \ge y]$, then by rearranging we have
\[ \chi(\lambda)\geq \frac12(\lambda_c-\lambda)^{-2}\P_{\lambda_c}[\Vol(\C) \ge y]. \]
\item [ 2.] If $(\lambda_c-\lambda)^2\chi(\lambda)\leq \frac12\P_{\lambda_c}[\Vol(\C) \ge y]$, then returning to \eqref{e:s3},
\[ \frac12\P_{\lambda_c}[\Vol(\C) \ge y] \le c_2 (\lambda_c-\lambda) \sqrt{\P_{\lambda_c}[\Vol(\C) \ge y] \chi(\lambda)}, \]
so that 
\[ \chi(\lambda)\geq \frac1{4c_2^2}(\lambda_c-\lambda)^{-2}\P_{\lambda_c}[\Vol(\C) \ge y]. \]
\end{enumerate}
Inequality \eqref{e:s2} is thus established with $c=\min(\frac12,\frac1{4c_2^2})$. 

\subsection{Proof of Theorem \ref{t:t}}
We begin with the proof of \eqref{e:t1}. First we recall from Corollary~\ref{c:vc} that $\theta(\lambda) = \P_\lambda[\Vol(\C) =  \infty]$ for every $\lambda > 0$; hence we may assume $\theta(\lambda_c) = 0$ because otherwise \eqref{e:t1} holds trivially. It also suffices to prove the result for $y$ sufficiently large, by adjusting constants.
 
For $y > 0$ define $\lambda = \lambda(y) > \lambda_c$ such that
\[ \P_\lambda[\Vol(\mathcal{C})\geq y]=2\P_{\lambda_c}[\Vol(\mathcal{C})\geq y] , \]
which exists by the continuity of $\lambda \to \P_\lambda[\Vol(\mathcal{C})\geq y]$ (see Corollary \ref{c:vc}). We claim that $\lambda(y) \to \lambda_c$ as $y \to \infty$. Indeed if $\limsup_{y \to \infty} \lambda(y) = \lambda_0 > \lambda_c$, then along a subsequence $y_k \to \infty$,
\[2\P_{\lambda_c}[\Vol(\mathcal{C})\geq y_k] = \P_\lambda[\Vol(\mathcal{C})\geq y_k] \ge  \P_{\lambda_0}[\Vol(\C) = \infty] = \theta(\lambda_0) > 0, \]
whereas we assume that, as $y \to \infty$,
\[  \P_{\lambda_c}[\Vol(\mathcal{C})\geq y] \to \theta(\lambda_c) = 0 , \]
a contradiction. Hence, in particular, there exist $y_0, c_0, \beta_0 > 0$ such that $\theta(\lambda) \ge c_0 (\lambda-\lambda_c)^{\beta_0}$ for all $y \ge y_0$. 

Now fix $y \ge y_0$ and consider the algorithm $\mathcal{T}$ for the event $A = \{ \Vol(\C) \ge y \}$. Applying the first statement of Proposition~\ref{p:entropic} (with $\lambda_1= \lambda_c$ and $\lambda_2 = \lambda$) we have
\begin{align*}
 \P_{\lambda_c}[\Vol(\mathcal{C})\geq y] & = \P_\lambda[\Vol(\mathcal{C})\geq y]-\P_{\lambda_c}[\Vol(\mathcal{C})\geq y] \\
&\quad \leq (2/\lambda)^{1/2} (\lambda-\lambda_c)\sqrt{\P_{\lambda}[\Vol(\mathcal{C})\geq y]\E_{\lambda_c}[\PVol(\mathcal{W}_\mathcal{A})]}  \\
& \quad \le 2 \lambda_c^{-1/2} (\lambda-\lambda_c)\sqrt{\P_{\lambda_c}[\Vol(\mathcal{C})\geq y]\E_{\lambda_c}[\PVol(\mathcal{W}_\mathcal{A})]} 
\end{align*}
and hence
\[    \P_{\lambda_c}[\Vol(\mathcal{C})\geq y] \le (4/\lambda_c) (\lambda-\lambda_c)^2  \E_{\lambda_c}[\PVol(\mathcal{W}_\mathcal{A})].  \]
Combining with
\[ 2 \P_{\lambda_c}[\Vol(\mathcal{C})\geq y] = \P_\lambda[\Vol(\mathcal{C})\geq y] \ge \P_\lambda[\Vol(\C) = \infty] = \theta(\lambda) \ge c_0 (\lambda-\lambda_c)^{\beta_0} , \]
yields
\[ \P_{\lambda_c}[\Vol(\mathcal{C})\geq y]^{2/\beta_0-1}  \E_{\lambda_c}[\PVol(\mathcal{W}_\mathcal{A})]  \ge (\lambda_c/4) (2/c_0)^{2/\beta_0} . \]
The result then follows by the second statement of Proposition \ref{p:vc3}.

For \eqref{e:t2} we fix $\lambda \le \lambda_c$ and $y \ge 1$ and again consider the algorithm $\mathcal{T}$ for the event $A = \{ \Vol(\C) \ge y \}$. This time applying the second statement of Proposition \ref{p:entropic} (with $\lambda_1 = \lambda_c$ and $\lambda_2 = \lambda$) and combining with the second statement of Proposition \ref{p:vc3}, there exists a constant $c = c(\lambda_0) > 0$ such that, for all $\lambda \in [\lambda_0, \lambda_c]$,
\begin{align*}
 \log \P_{\lambda_c}[ \Vol(\C) \ge y] -   \log \P_{\lambda}[ \Vol(\C) \ge y]  & \le    \frac{ \lambda^{-1} (\lambda_c-\lambda)^2 \E_{\lambda_c} [ \PVol (\mathcal{W}_\mathcal{A}) ]  }{ \P_{\lambda_c}[ \Vol(\C) \ge y] }  + 1  \\
&    \le    \frac{  c \lambda^{-1} (\lambda_c-\lambda)^2   \int_0^y \P_{\lambda_c}[\Vol(\C) \ge u] du    }{ \P_{\lambda_c}[ \Vol(\C) \ge y] }  + 1 
\end{align*}
as required.

\subsection{Proof of Theorem \ref{t:m}}
We start with the proof of \eqref{e:m1}, which is similar to the proof of \eqref{e:s2}. By adjusting constants, it is enough to prove the result for $\lambda < \lambda_c$ sufficiently close to $\lambda_c$. Recall from \eqref{e:mag} that the magnetisation can be defined as $M(\rho) =  \P_{\lambda_c,\rho}[0\longleftrightarrow\mathcal{G}]$. Fix $\lambda < \lambda_c$ and consider the algorithm $\mathcal{T}$ for the event $A = \{ 0 \longleftrightarrow \mathcal{G}\}$. Applying the first statement of Proposition \ref{p:entropic2} (with $\lambda_1 = \lambda$, $\lambda_2 = \lambda_c$, and $\rho = (\lambda_c-\lambda)^2$), and then using the first statement of Proposition \ref{p:vc3}, gives that  
\begin{align*}
 M((\lambda_c-\lambda)^2)-\P_{\lambda,(\lambda_c-\lambda)^2}[0\longleftrightarrow\mathcal{G}]   &   \le \sqrt{2/\lambda_c} (\lambda_c-\lambda)\sqrt{M((\lambda_c-\lambda)^2) \E_\lambda[ \PVol( \mathcal{W}_\mathcal{A} ) ] } \\
 & \le c_1 (\lambda_c-\lambda)\sqrt{M((\lambda_c-\lambda)^2) \chi(\lambda)  }
 \end{align*}
 for some constant $c_1 > 0$ and $\lambda$ sufficiently close to $\lambda_c$. By Markov's inequality and using the independence of the ghost field $\mathcal{G}$,
\begin{align*}
\P_{\lambda,(\lambda_c-\lambda)^2}[0\longleftrightarrow\mathcal{G}] & =  \P_{\lambda,(\lambda_c-\lambda)^2}[ | \C\cap\mathcal{G}|  \ge 1 ]  \le \E_{\lambda,(\lambda_c-\lambda)^2}[ | \C\cap\mathcal{G}| ]   \\ 
& = (\lambda_c-\lambda)^2\E_\lambda [ \Vol (\C) ] =(\lambda_c-\lambda)^2\chi(\lambda). 
\end{align*}
Combining gives
\[ M((\lambda_c-\lambda)^2)- (\lambda_c-\lambda)^2\chi(\lambda) \le c_1 (\lambda_c-\lambda)\sqrt{M((\lambda_c-\lambda)^2)  \chi(\lambda) } \]
and applying the same disjunction as in the proof of \eqref{e:s2} (with $M((\lambda_c-\lambda)^2)$ replacing $\P_{\lambda_c}[\Vol(\C) \ge y]$) gives the result. 

Next, the proof of (\ref{e:m2}), which is similar to the proof of \eqref{e:t1}. First observe that we may assume $\lim_{\rho \to 0} M(\rho) = 0$, since otherwise the result is trivial. Also, by adjusting constants, it is enough to prove the result for sufficiently small $\rho$. Define $ \lambda(\rho) >\lambda_c$ such that
\[ \P_{\lambda(\rho),\rho}[0 \longleftrightarrow \mathcal{G}]=  2 \P_{\lambda_c,\rho}[0 \longleftrightarrow \mathcal{G}]   \]
which exists by the continuity of $\lambda \to \P_{\lambda,\rho}[0 \longleftrightarrow \mathcal{G}]$ (see Corollary \ref{c:vc}).  We claim that $\lambda(\rho) \to \lambda_c$ as $\rho \to 0$. Indeed if $\limsup_{\rho \to 0} \lambda(\rho) = \lambda_0 > \lambda_c$ then along a subsequence $\rho_k \to 0$
\begin{equation}
\label{e:m3} 2 M(\rho_k)   \ge  2   \P_{\lambda_c,\rho_k}[0 \longleftrightarrow \mathcal{G}]    = \P_{\lambda(\rho_k),\rho_k}[0 \longleftrightarrow \mathcal{G}]  .
\end{equation}
Recall from Corollary~\ref{c:vc} that $\theta(\lambda) = \P_{\lambda}[ \Vol(\C) = \infty]$, which by the independence of the ghost field $\mathcal{G}$ implies that 
\[   \P_{\lambda,\rho}[0\longleftrightarrow\mathcal{G}] \ge \theta(\lambda)   \]
for all $\lambda,\rho> 0$. Hence \eqref{e:m3} violates our assumption that $M(\rho) \to 0$. In particular, there exist $\rho_0, c_0, \beta_0 > 0$ such that $ \P_{\lambda(\rho),\rho}[0\longleftrightarrow\mathcal{G}]  \ge \theta(\lambda(\rho)) \ge c_0 (\lambda(\rho)-\lambda_c)^{\beta_0}$ for all $\rho \in (0,\rho_0)$.

Now fix $\rho \in (0,\rho_0)$ and consider the algorithm $\mathcal{T}$ for the event $A = \{ 0 \longleftrightarrow \mathcal{G}\}$. As in the proof of \eqref{e:t1}, applying the first statement of Proposition~\ref{p:entropic} (in its entropic form in Proposition \ref{p:entropic2}, with $\lambda_1= \lambda_c$ and $\lambda_2 = \lambda(\rho)$) we have
\begin{align*}
 \P_{\lambda_c, \rho}[0 \longleftrightarrow \mathcal{G}] & = \P_{\lambda(\rho),\rho}[0 \longleftrightarrow \mathcal{G}]-\P_{\lambda_c, \rho}[0 \longleftrightarrow \mathcal{G}] \\
 &\quad \leq (2/\lambda(\rho))^{1/2} (\lambda(\rho)-\lambda_c)\sqrt{\P_{\lambda(\rho), \rho}[0 \longleftrightarrow \mathcal{G} ]\E_{\lambda_c,\rho}[\PVol(\mathcal{W}_\mathcal{A})]}  \\
& \quad \le 2 \lambda_c^{-1/2} (\lambda(\rho)-\lambda_c)\sqrt{\P_{\lambda_c,\rho}[0 \longleftrightarrow \mathcal{G}]\E_{\lambda_c,\rho}[\PVol(\mathcal{W}_\mathcal{A})]} 
\end{align*}
and hence
\[    \P_{\lambda_c,\rho}[0 \longleftrightarrow \mathcal{G}] \le (4/\lambda_c) (\lambda(\rho)-\lambda_c)^2  \E_{\lambda_c,\rho}[\PVol(\mathcal{W}_\mathcal{A})].  \]
Combining with
\[ 2 \P_{\lambda_c,\rho}[0 \longleftrightarrow \mathcal{G}] = \P_{\lambda(\rho),\rho}[0 \longleftrightarrow \mathcal{G}] \ge \theta(\lambda(\rho)) \ge c_0 (\lambda(\rho)-\lambda_c)^{\beta_0}  \]
yields
\[ \P_{\lambda_c,\rho}[0 \longleftrightarrow \mathcal{G}]^{2/\beta_0-1}  \E_{\lambda_c}[\PVol(\mathcal{W}_\mathcal{A})]  \ge (\lambda_c/4) (2/c_0)^{2/\beta_0} . \]
Applying the third statement of Proposition~\ref{p:vc3}, we deduce the existence of a constant $c_2>0$ such that, for all $\rho$ sufficiently small,
\[   \P_{\lambda_c,\rho}[0 \longleftrightarrow \mathcal{G}]^{2/\beta_0-1}  \rho^{-1}   \E_{\lambda_c}\big[1 - e^{-\rho \Vol(\mathcal{C})} \big] \ge c_2. \]  
Recalling from \eqref{e:mag} that 
\[ M(\rho) =\E_{\lambda_c}\big[1 - e^{-\rho \Vol(\mathcal{C})} \big] =  \P_{\lambda_c,\rho}[0 \longleftrightarrow \mathcal{G}]   \] 
we conclude that  $M(\rho)^{2/\beta_0} \rho^{-1} \ge c_2$ as required.

\bigskip
\bibliographystyle{plain}
\bibliography{paper}

\begin{thebibliography}{10}

\bibitem{att18}
D.~Ahlberg, V.~Tassion, and A.~Teixeira.
\newblock Sharpness of the phase transition for continuum percolation.
\newblock {\em Probab. Theory Related Fields}, 172(1--2):525--281, 2018.

\bibitem{ab87}
M.~Aizenman and D.J. Barsky.
\newblock Sharpness of the phase transition in percolation models.
\newblock {\em Comm. Math. Phys.}, 108(3):489--526, 1987.

\bibitem{an84}
M.~Aizenman and C.M. Newman.
\newblock Tree graph inequalities and critical behavior in percolation models.
\newblock {\em J. Stat. Phys.}, 36:107--143, 1984.

\bibitem{chs21}
S.~Chatterjee, J.~Hanson, and P.~Sosoe.
\newblock Subcritical connectivity and some exact tail exponents in high
  dimensional percolation.
\newblock {\em arXiv preprint arXiv:2107.14347}, 2021.

\bibitem{cc87}
J.T. Chayes and L.~Chayes.
\newblock The mean field bound for the order parameter of {B}ernoulli
  percolation.
\newblock In H.~Kesten, editor, {\em Percolation Theory and Ergodic Theory of
  Infinite Particle Systems. The IMA Volumes in Mathematics and its
  Applications, vol. 8}, pages 49--71. Springer, New York, 1987.

\bibitem{dem10}
A.~Dembo and O.~Zeitouni.
\newblock {\em Large Deviations Techniques and Applications}.
\newblock Springer Berlin, Heidelberg, 2010.

\bibitem{dm21}
V.~Dewan and S.~Muirhead.
\newblock Upper bounds on the one-arm exponent for dependent percolation
  models.
\newblock {\em Probab. Theory Related Fields}, 185(1--2):41--88, 2023.

\bibitem{drt20}
H.~Duminil-Copin, A.~Raoufi, and V.~Tassion.
\newblock Subcritical phase of $d$-dimensional {P}oisson-{B}oolean percolation
  and its vacant set.
\newblock {\em Ann. H. Lebesgue}, 3:677--700, 2020.

\bibitem{dn85}
R.~Durrett and B.~Nguyen.
\newblock Thermodynamic inequalities for percolation.
\newblock {\em Commun. Math. Phys.}, 99:253--269, 1985.

\bibitem{fh17}
R.~Fitzner and R.~van~der Hofstad.
\newblock Mean-field behavior for nearest-neighbor percolation in $d > 10$.
\newblock {\em Electron. J. Probab.}, 22:65 pp., 2017.

\bibitem{gil61}
E.N. Gilbert.
\newblock Random plane networks.
\newblock {\em J. Soc. Indust. Appl. Math.}, 9:533--543, 1961.

\bibitem{gou08}
J.-B. Gour\'{e}r\'{e}.
\newblock Subcritical regimes in the {P}oisson {B}oolean model of continuum
  percolation.
\newblock {\em Ann. Probab.}, 36(4):1209--1220, 2008.

\bibitem{gri99}
G.R. Grimmett.
\newblock {\em Percolation}.
\newblock Springer, 1999.

\bibitem{hal85}
P.~Hall.
\newblock On continuum percolation.
\newblock {\em Ann. Probab.}, 13(4):1250--1266, 1985.

\bibitem{hs90}
T.~Hara and G.~Slade.
\newblock Mean-field critical behaviour for percolation in high dimensions.
\newblock {\em Comm. Math. Phys.}, 128:333--391, 1990.

\bibitem{hhlm19}
M.~Heydenreich, R.~van~der Hofstad, G.~Last, and K.~Matzke.
\newblock Lace expansion and mean-field behaviour for the random connection
  model.
\newblock {\em arXiv preprint arXiv:1908.11356}, 2019.

\bibitem{hut20}
T.~Hutchcroft.
\newblock Slightly supercritical percolation on nonamenable graphs {I}: The
  distribution of finite clusters.
\newblock {\em arXiv preprint arXiv:2002.02916}, 2020.

\bibitem{hut21}
T.~Hutchcroft.
\newblock On the derivation of mean-field percolation critical exponents from
  the triangle condition.
\newblock {\em J. Stat. Phys.}, 189(6), 2022.

\bibitem{hms21}
T.~Hutchcroft, E.~Michta, and G.~Slade.
\newblock High-dimensional near-critical percolation and the torus plateau.
\newblock {\em Ann. Probab. (to appear)}.

\bibitem{kul78}
S.~Kullback.
\newblock {\em Information theory and statistics}.
\newblock Dover, 1978.

\bibitem{lpz17}
G.~Last, M.D. Penrose, and S.~Zuyev.
\newblock On the capacity functional of the infinite cluster of a {B}oolean
  model.
\newblock {\em Ann. Appl. Probab.}, 27(3):1678--1701, 2017.

\bibitem{lee93}
S.~Lee.
\newblock An inequality for greedy lattice animals.
\newblock {\em Ann. Appl. Probab.}, 3(4):1170--1188, 1993.

\bibitem{lss97}
T.M. Liggett, R.H. Schonmann, and A.M. Stacey.
\newblock Domination by product measures.
\newblock {\em Ann. Probab.}, 25(1):71--95, 1997.

\bibitem{mr08}
R.~Meester and R.~Roy.
\newblock {\em Continuum percolation}.
\newblock Cambridge University Press, 2008.

\bibitem{new86}
C.M. Newman.
\newblock Some critical exponent inequalities for percolation.
\newblock {\em J. Stat. Phys.}, 45:359--368, 1986.

\bibitem{new87}
C.M. Newman.
\newblock Another critical exponent inequality for percolation: $\beta \ge
  2/\delta$.
\newblock {\em J. Stat. Phys.}, 47:695--699, 1987.

\bibitem{pen03}
M.D. Penrose.
\newblock {\em Random Geometric Graphs}.
\newblock Oxford University Press, 2003.

\bibitem{zie18}
S.~Ziesche.
\newblock Sharpness of the phase transition and lower bounds for the critical
  intensity in continuum percolation on $\mathbb{R}^d$.
\newblock {\em Ann. Inst. H. Poincar\'{e} Probab. Statist.}, 54(2):866--878,
  2018.

\end{thebibliography}

\end{document}